\documentclass[11pt,reqno]{amsart}
\pdfoutput=1
\usepackage[margin=1.2in,marginparwidth=1.5cm,marginparsep=0.5cm]{geometry}
\usepackage{amsfonts}
\usepackage{amsmath}
\usepackage{amssymb}
\usepackage{mathrsfs}
\usepackage{geometry}
\usepackage{graphicx}
\usepackage{subfigure}
\usepackage{cite}
\usepackage{hyperref}
\usepackage{microtype}
\usepackage{enumerate}
\usepackage{mathrsfs}
\usepackage{tikz}%% drawing
\usepackage{tikz-cd}%% drawing
\usepackage{color}
\usepackage{ragged2e}
 \usepackage{float}
\allowdisplaybreaks[1]
\numberwithin{equation}{section}
\usepackage{needspace}

\newcommand*{\circled}[1]{\lower.7ex\hbox{\tikz\draw (0pt, 0pt)%
    circle (.5em) node {\makebox[1em][c]{\small #1}};}}

\newcommand{\al}{\alpha}
\newcommand{\be}{\beta}

\newcommand{\ve}{\varepsilon}

\newcommand{\R}{\mathbb{R}}
\newcommand{\N}{\mathbb{N}}

\newcommand{\Q}{\mathbb{Q}}
\newcommand{\Z}{\mathbb{Z}}
\newcommand{\T}{\mathbb{T}}

\newcommand{\n}[1]{\Vert #1\Vert }

\newcommand{\bbn}[1]{\Big\Vert #1 \Big \Vert }
\newcommand{\lr}[1]{\left\{ #1\right\} }
\newcommand{\lrc}[1]{\left[ #1\right] }
\newcommand{\lrs}[1]{\left( #1\right) }

\newcommand{\abs}[1]{|#1| }

\newcommand{\bbabs}[1]{\Big | #1 \Big| }

\newcommand{\pa}{\partial}

\newtheorem{theorem}{Theorem}[section]
\newtheorem{lemma}[theorem]{Lemma}
\theoremstyle{definition}
\newtheorem{definition}[theorem]{Definition}

\newtheorem{remark}[theorem]{Remark}

\numberwithin{equation}{section}
\newtheorem{proposition}[theorem]{Proposition}
\newtheorem{corollary}[theorem]{Corollary}

\newtheorem*{ackno}{Acknowledgements}

\newcommand{\ind}{\mathbf 1}
\newcommand{\bfone}{\ind}

\newcommand{\ft}{\widehat}
\newcommand{\wh}{\ft}

\newcommand{\QQ}{\mathcal{Q}}
\newcommand{\RR}{\mathcal{R}}
\newcommand{\cQ}{\QQ}
\newcommand{\cR}{\RR}

\newcommand{\lesssimab}{\lesssim_{a,b}}
\newcommand{\gtrsimab}{\gtrsim_{a,b}}
\newcommand{\simab}{\sim_{a,b}}

\makeatletter
\@namedef{subjclassname@2020}{%
  \textup{2020} Mathematics Subject Classification}
\makeatother

\begin{document}
\baselineskip=14pt

\title[Sharp Strichartz estimates on irrational tori]
{Sharp Strichartz estimates on quadratic irrational tori}

\author[R. Liang, S. Shen, Y. Wang]{Rui Liang, Shunlin Shen, Yuzhao Wang}

\address{Rui Liang\\
School of Mathematical Sciences\\
South China Normal University\\
Guangzhou, Guangdong, 510631, China}
\email{liangrui@m.scnu.edu.cn}

% \author{Shunlin Shen}
\address{Shunlin Shen,
School of Mathematical Sciences, University of Science and Technology of China\\
Hefei, 230026\\
China}
\email{slshen@ustc.edu.cn}

\address{School of Mathematics, University of Birmingham, Watson Building, Edgbaston, Birmingham B15 2TT, United Kingdom}
\email{y.wang.14@bham.ac.uk}



\begin{abstract}
We prove sharp $L^{4}$ Strichartz estimates for two-dimensional periodic
Schr\"{o}dinger flows with integral diagonal symbols $ak_{1}^2+bk_2^2$ where $a,b \in \Z\setminus \{0\}$.
The estimates exhibit a sharp arithmetic dichotomy according to whether
$-ab$ is a square: 
for frequencies bounded by $N$, the optimal loss is $(\log N)^{1/4}$ when $-ab$ is not a square and $N^{1/4}$
when $-ab$ is a square.
As a consequence, we show that arbitrarily small perturbations
of the rectangular aspect ratio can lead to a transition
between global well-posedness and norm inflation
for cubic hyperbolic NLS.
\end{abstract}

\maketitle
% {\small
% \tableofcontents
% }

\section{Introduction}\label{section:Introduction}
\subsection{Elliptic and hyperbolic NLS on rectangular tori}\label{section:Elliptic and hyperbolic NLS}
In this paper,
we study how the arithmetic of a rectangular torus affects Strichartz estimates and the well-posedness theory for elliptic and hyperbolic nonlinear Schr\"{o}dinger equations.

Let 
\begin{align*}
\T^{2}_{\al_{1},\al_{2}}=
(\R/2\pi\al_{1}\Z)\times(\R/2\pi\al_{2}\Z), \quad \al_{1},\al_{2}>0,
\end{align*}
be the rectangular torus.
We call the torus \emph{rational}
if $\al_{1}/\al_{2}\in\Q$, and \emph{irrational} otherwise.
On this torus, the cubic elliptic nonlinear Schr\"odinger equation is
\begin{equation}\label{equ:ENLS}
\begin{cases}
 i\pa_{t} u+\lrs{\pa_{y_{1}}^{2}+\pa_{y_{2}}^{2}}u=\mu |u|^{2}u,\\
 u(0)=u_0,
\end{cases}
\quad (y_{1},y_{2})\in\T^{2}_{\al_{1},\al_{2}},
\end{equation}
whereas its hyperbolic counterpart is
\begin{equation}
%\tag{HNLS}
\label{equ:HNLS}
\begin{cases}
i\pa_tu+(\pa_{y_{1}}^{2}-\pa_{y_{2}}^{2})u =\mu |u|^{2}u,\\
        u(0)=u_0,
\end{cases}
\quad (y_{1},y_{2})\in\T^{2}_{\al_{1},\al_{2}},
\end{equation}
Here, $\mu\in\lr{\pm1}$.  
Of particular interest in this work are rectangular tori with
specific aspect ratios $\alpha_1/\alpha_2$.

\begin{definition}[Quadratic irrational rectangular tori]
\label{def:quadratic irrational}
In this paper, a rectangular torus $\T^{2}_{\al_{1},\al_{2}}$ is called
\emph{quadratic irrational} if its aspect ratio
$\rho=\al_{1}/\al_{2}$ satisfies
\begin{align*}
\rho\notin\Q,\quad \rho^{2}\in\Q.
\end{align*}
Equivalently, $\rho=\sqrt{p/q}$, where $p,q$ are coprime positive
integers and $pq$ is not a square.
\end{definition}

Throughout this paper, \emph{quadratic irrational} is used in the
restricted sense of Definition~\ref{def:quadratic irrational}.
The significance of this family is that the equations can be normalized
to have integral dispersion coefficients of the form $ak_{1}^2+bk_2^2$ where $a,b \in \Z\setminus \{0\}$, 
while the hyperbolic symbol when $ab < 0$ has no nonzero lattice null vector. 
We explain this reduction in
Subsection~\ref{section:From rectangular tori}.

\subsection{Background and motivation}
On $\R^{2}$, elliptic and hyperbolic Schr\"odinger propagators have the
same dispersive decay and the same admissible Strichartz indexes.
Both cubic equations are mass-critical by Euclidean scaling.
On compact manifolds, global Euclidean dispersive estimates generally
fail. Burq, G\'erard, and Tzvetkov \cite{BGT} established Strichartz
estimates with derivative loss on compact Riemannian manifolds.
On flat tori, these estimates are closely tied to arithmetic counting
properties of the frequency lattice.

For the elliptic equation \eqref{equ:ENLS} on flat tori, Bourgain's work
\cite{BO93} initiated the systematic use of lattice counting, while
Bourgain--Demeter decoupling \cite{BD15} later yielded essentially sharp
$N^{\ve}$-type bounds for many Strichartz estimates.
At the two-dimensional $L^{4}$ endpoint, Herr--Kwak \cite{HK24}
proved the sharp square-torus estimate, for $N\ge2$,
\begin{align}\label{equ:Strichartz}
\n{e^{it\Delta}P_{\le N}\phi}_{L^{4}([0,2\pi]\times\T^{2})}\lesssim (\log N)^{1/4}\n{\phi}_{L^{2}(\T^{2})}.
\end{align}
Here and below, $P_{\le N}$ denotes the sharp Fourier projection to
$[-N,N]^{2}\cap\Z^{2}$. The logarithm is sharp by the standard box
example; see \cite{BO93,TT01,Kishimoto}.
Herr--Kwak \cite{HK24} also proved a lossless estimate on a logarithmic
time scale, from which they obtained small-mass global well-posedness
for the cubic elliptic equation in $H^{s}(\T^{2})$, for every $s>0$. 
Their subsequent work \cite{HK26} establishes global well-posedness in
the same positive Sobolev range for arbitrary defocusing data and for
focusing data with mass below that of the Euclidean ground state.

For the hyperbolic equation \eqref{equ:HNLS}, rational null
lines create an additional obstruction. For example, on the
standard square torus, the symbol $k_1^2-k_2^2$ vanishes along
the rational lines $k_1=\pm k_2$, so Fourier modes supported
on these lines exhibit no temporal oscillation under the
Schr\"odinger flow. 
Wang \cite{Wang13} proved the sharp estimate, for $N\ge2$,
\begin{align}\label{equ:hyper Strichartz}
\n{e^{it\Box}P_{\le N}\phi}_{L^{4}([0,2\pi]\times\T^{2})}\lesssim N^{1/4}\n{\phi}_{L^{2}(\T^{2})},
\end{align}
where $\Box=\pa_{y_{1}}^{2}-\pa_{y_{2}}^{2}$ is the hyperbolic Laplacian. The sharp $N^{1/4}$ loss reflects the fact that Fourier
modes supported on either null line undergo no temporal oscillation.
Using \eqref{equ:hyper Strichartz}, Wang \cite{Wang13} established local well-posedness
in $H^{s}$ for $s>\frac12$. More recently, Shen--Wang \cite{ShenWang} proved norm inflation for $0<s\le\frac12$, including the threshold regularity $H^{1/2}(\T^{2})$. Thus, the square-torus hyperbolic threshold is $s=\frac12$, whereas the elliptic local theory extends to every $H^{s}$, $s>0$. For further developments on hyperbolic
and more general non-elliptic Schr\"odinger equations, see for example
\cite{GS93,SW24,GT12,BD17,BOW25,BW25,LZ25,LZ26} and the references therein.

The contrast between \eqref{equ:Strichartz} and
\eqref{equ:hyper Strichartz} raises the question of whether
rational null directions of $ak_1^2+bk_2^2$ are the only
obstruction to Strichartz estimates with a sharp logarithmic loss. 
We answer this affirmatively for $a,b \in \Z \setminus \{0\}$; see
Subsection~\ref{section:Sharp estimates and nonlinear consequences}.
These results show that the distinction between the elliptic
and hyperbolic theories on the square torus can change under
arbitrarily small perturbations of the rectangular aspect ratio
(see Subsection~\ref{section:Domain sensitivity}).
The relation of this integral-coefficient setting to long-time
Strichartz estimates \cite{DGG17,DGGRM22,GMO26} is discussed
in Subsection~\ref{section:Comparison}.

\subsection{Main results}\label{section:Main results}
We first normalize the equation, then state the sharp linear estimates
and their nonlinear and geometric consequences.

\subsubsection{From rectangular tori to quadratic forms}\label{section:From rectangular tori}
Assume that the rectangular torus has rational squared aspect ratio:
\begin{equation}\label{equ:qirr}
\left(\frac{\al_1}{\al_2}\right)^2=\frac{p}{q},
\quad p,q\in\N,\quad (p,q)=1.
\end{equation}
Under the change of variables $y_j=\al_jx_j$, a solution $u$
of \eqref{equ:ENLS} or \eqref{equ:HNLS} gives rise to
$v(t,x)=u(t,\al_1x_1,\al_2x_2)$, which satisfies
\begin{align}\label{equ:NLS,v}
i\pa_t v-\Omega_\alpha^\sigma v=\mu|v|^2v
\end{align}
on $\T^2=(\R/2\pi\Z)^2$.
Here, $\sigma=1$ and $\sigma=-1$ correspond to the elliptic
and hyperbolic equations, respectively, and
$\Omega_\alpha^\sigma$ is the Fourier multiplier with symbol
\begin{equation}\label{equ:rect}
\omega_\alpha^\sigma(k)
=\al_1^{-2}k_1^2+\sigma\al_2^{-2}k_2^2
=\frac{1}{q\al_1^2}\bigl(qk_1^2+\sigma pk_2^2\bigr).
\end{equation}
Setting $c=(q\al_1^2)^{-1}$, $\tau=ct$, and
$v(t,x)=c^{1/2}w(\tau,x)$, we obtain
\begin{align}\label{NLS2}
i\pa_\tau w-\Omega_{p,q}^\sigma w=\mu|w|^2w,
\end{align}
where $\Omega_{p,q}^\sigma$ is the Fourier multiplier with
symbol $qk_1^2+\sigma pk_2^2$.
Thus, fixed rescalings of space, time, and amplitude reduce
both equations to a form with integral dispersion coefficients.
In particular, the quadratic irrational tori of
Definition~\ref{def:quadratic irrational} have rationally
commensurate dispersion coefficients.

Motivated by this reduction, we henceforth work on standard square torus $\T^{2}$ and consider
the slightly more general cubic equation
\begin{equation}
\label{equ:cubic NLS}
 \begin{cases}
        i\pa_{t}u -\Omega_{a,b} u
        =\mu |u|^{2}u,\\
        u(0)=u_0,
\end{cases}
\quad x=(x_{1},x_{2})\in\T^{2} =(\R/2\pi\Z)^2,
\end{equation}
where $\mu\in\{\pm1\}$.
Here $\Omega_{a,b}$ denotes the Fourier multiplier with integral
diagonal symbol
\begin{equation}
\label{equ:disp}
        \omega_{a,b}(k)=a k_{1}^{2}+b k_{2}^{2},
        \quad a,b\in\Z\setminus\{0\}.
\end{equation}
The associated linear propagator is
\begin{equation}
\label{equ:flow}
        S_{a,b}(t)=e^{-it\Omega_{a,b}},
        \quad
        \widehat{S_{a,b}(t)\phi}(k)=e^{-it\omega_{a,b}(k)}\wh\phi(k).
\end{equation}
The propagator is $2\pi$-periodic in time because its symbol is
integer-valued. For fixed $a,b$, we write
$\omega,\Omega,S(t)$ for $\omega_{a,b},\Omega_{a,b},S_{a,b}(t)$.
Nonzero rational coefficients are also covered after clearing
denominators and rescaling time and amplitude.

\begin{definition}[Square and nonsquare cases]
\label{def:sq}
For the integral diagonal form $\omega_{a,b}$ in \eqref{equ:disp}, we say
that $(a,b)$ lies in the \emph{square case} if
\begin{align*}
-ab=m^{2}\quad\text{for some }m\in\Z.
\end{align*}
Otherwise, we say that $(a,b)$ lies in the \emph{nonsquare case}.
\end{definition}

The square case occurs precisely for hyperbolic forms with a nonzero
lattice null vector, or equivalently a rational null line. See
Lemma~\ref{lemma:null}. Every elliptic form belongs to the nonsquare case.

\subsubsection{Sharp estimates and nonlinear consequences}
\label{section:Sharp estimates and nonlinear consequences}
We begin with the sharp linear estimates, which exhibit the arithmetic dichotomy introduced in Definition~\ref{def:sq}.
\begin{theorem}[Sharp $L^{4}$ Strichartz dichotomy]\label{thm:sharp,Strichartz,L4}
Let $a,b\in\Z\setminus\{0\}$, and let $S(t)=S_{a,b}(t)$ be given by \eqref{equ:flow}.
Then, for every dyadic $N\ge2$,
\begin{equation}\label{equ:sharp,Strichartz,L4}
\n{S(t)P_{\le N}\phi}_{L^{4}([0,2\pi]\times\T^{2})} \lesssimab \Theta(N)^{1/4}\n{\phi}_{L^{2}(\T^{2})},
\end{equation}
where
\begin{align*}
\Theta(N)=
\begin{cases}
 N, & -ab \text{ is a square},\\
 \log N, & -ab \text{ is not a square}.
\end{cases}
\end{align*}
Both losses are sharp, up to constants depending only on $a$, $b$.
\end{theorem}

The fixed-time bounds in \eqref{equ:sharp,Strichartz,L4} follow by summing the following short-time estimates.
\begin{theorem}[Sharp short-time $L^{4}$ Strichartz estimates]\label{thm:sharp,short-time,L4}
Let $a,b\in\Z\setminus\{0\}$, and let $S(t)=S_{a,b}(t)$ be given by \eqref{equ:flow}.
There exists $c=c(a,b)>0$ such that, for every dyadic $N\ge2$,
\begin{equation}\label{equ:sharp,short-time,L4}
\n{S(t)P_{\le N}\phi}_{L^{4}([0,T_{N}]\times\T^{2})}\lesssimab \n{\phi}_{L^{2}(\T^{2})}.
\end{equation}
Here
\begin{align*}
T_{N}=
\begin{cases}
 cN^{-1}, & -ab \text{ is a square},\\
 c(\log N)^{-1}, & -ab \text{ is not a square}.
\end{cases}
\end{align*}
\end{theorem}

\begin{remark} \rm
The fixed-time estimate \eqref{equ:sharp,Strichartz,L4} for $(a,b)=(1,-1)$ is due to Wang \cite[Theorem 2.1]{Wang13}, while both \eqref{equ:sharp,Strichartz,L4} and \eqref{equ:sharp,short-time,L4} for $(a,b)=(1,1)$ are due to Herr--Kwak\cite{HK24}. In the square case, i.e. $-ab$ is a square, the fixed-time
estimate \eqref{equ:sharp,Strichartz,L4} for general $(a,b)$ can be obtained by the rational-torus covering
argument of Guo--Oh--Wang \cite[Subsection 1.2]{GOW14}. 
We prove the sharp estimates of \eqref{equ:sharp,Strichartz,L4} and
\eqref{equ:sharp,short-time,L4} in the remaining cases in
Section~\ref{section:Sharp Strichartz}. The main extension
concerns nonsquare integral forms, particularly those of
hyperbolic signature.
\end{remark}

Both time scales $T_N$ in \eqref{equ:sharp,short-time,L4} are optimal up to constants. Indeed, the null-line example
in Lemma~\ref{lemma:null-lb} forces $T_{N}\lesssim N^{-1}$ in the square
case, while Lemma~\ref{lemma:box} and a time-partition argument force
$T_{N}\lesssim(\log N)^{-1}$ in the nonsquare case.

\medskip

In the nonsquare case, the logarithmic time scale permits the
Herr--Kwak iteration \cite{HK24} for the cubic nonlinearity and gives
the following small-mass global theory.

\begin{theorem}[Small-mass global well-posedness]\label{thm:gwp}
Assume that $-ab$ is not a square. For every $s>0$, there exists $\delta=\delta(s,a,b)>0$ such that, for every $u_0\in H^{s}(\T^{2})$ with
\begin{align*}
 \n{u_0}_{L^{2}(\T^{2})}\le\delta,
\end{align*}
the cubic equation \eqref{equ:cubic NLS} has a solution
\begin{align*}
 u\in C(\R;H^{s}(\T^{2}))\cap Y^{s}_{\mathrm{loc}}(\R).
\end{align*}
Here, $Y^{s}_{\mathrm{loc}}$ denotes the adapted restriction space defined in
Subsection $\ref{section:Y spaces}$. The solution is unique in this class and conserves mass. For every $T>0$, the solution map into $C([-T,T];H^{s}(\T^{2}))$ is Lipschitz on bounded $H^{s}$-sets of initial data satisfying the stated mass bound.
\end{theorem}

Undoing the normalization in \eqref{equ:rect} gives the corresponding geometric
consequence for quadratic irrational rectangular tori.
\begin{corollary}[Quadratic irrational rectangular tori]\label{cor:quadratic,irrational,tori}
Let $\T^2_{\al_1,\al_2}$ be a quadratic irrational rectangular
torus in the sense of Definition~\ref{def:quadratic irrational}.
Then, for every $s>0$, both the cubic elliptic equation
\eqref{equ:ENLS} and the cubic hyperbolic equation \eqref{equ:HNLS} on
$\T^{2}_{\al_{1},\al_{2}}$ are globally well-posed in $H^{s}$ for
initial data with sufficiently small $L^{2}$-norm.
\end{corollary}
\begin{proof}
We write $(\al_{1}/\al_{2})^{2}=p/q$ in lowest terms. By \eqref{equ:rect}, the normalized coefficients are $a=q$ and $b=p$ in the elliptic case, so $-ab=-pq<0$ is not a square.
In the hyperbolic case, $a=q$ and $b=-p$, and the irrationality of $\al_{1}/\al_{2}$ implies that $-ab=pq$ is not a square. Thus, Theorem~\ref{thm:gwp} applies in both cases. Undoing the fixed
spatial, time, and amplitude changes proves the claim.
\end{proof}

\begin{remark}\rm
\label{rem:novelty}
To our knowledge, Corollary~\ref{cor:quadratic,irrational,tori} provides the first
small-mass global well-posedness result for the cubic hyperbolic NLS
in the fully periodic two-dimensional setting. It holds on quadratic
irrational rectangular tori in every $H^{s}$, $s>0$.
In the semi-periodic setting, Deng--Fan--Zhao
\cite[Theorems 1.1 and 1.3]{DFZ25} recently proved a lossless $L^{4}$
Strichartz estimate on $[0,1]\times\R\times\T$ and small-data global
well-posedness for the cubic hyperbolic NLS in $L^{2}(\R\times\T)$.

The elliptic conclusion is also new for this family, extending
the result of Herr--Kwak \cite{HK24} from the standard square
torus to quadratic irrational rectangular tori.
\end{remark}

\subsubsection{Domain sensitivity and endpoint regularity}\label{section:Domain sensitivity} 
We first explain how the standard square-torus theory applies to every form in the square case, i.e. $-ab$ is a square number. Guo--Oh--Wang \cite[Subsection 1.2]{GOW14} periodically extend functions from a rational rectangular torus to a
larger square torus, then transfer estimates by scaling.
The same construction applies to the hyperbolic operator. After interchanging coordinates if necessary, we write
\begin{align*}
 a=dr_{1}^{2},\quad b=-dr_{2}^{2},\quad d=\gcd(|a|,|b|),\quad r_{1},r_{2}\in\N,
\end{align*}
as permitted by $-ab$ being a square. The map
\begin{align*}
 V(\tau,x)=d^{-1/2}u(\tau/d,r_{1}x_{1},r_{2}x_{2})
\end{align*}
transforms \eqref{equ:cubic NLS} into the standard equation $i\pa_{\tau} V+\Box V=\mu|V|^{2}V$.
The additional spatial periods of $V$ are preserved by uniqueness, so the square-torus local theory pulls back to $u$. Conversely, any standard square-torus solution $V$ gives a solution
of \eqref{equ:cubic NLS} through
\begin{align*}
u(t,x)=\lambda^{1/2}V(\lambda t,r_{2}x_{1},r_{1}x_{2}),\quad \lambda=dr_{1}^{2}r_{2}^{2}.
\end{align*}
These fixed maps compare Sobolev norms up to constants depending on $a,b,s$. The reverse embedding therefore transfers the ill-posedness examples as well.

Consequently, throughout the square case, i.e.\ $-ab$ is a square, Wang's local well-posedness result \cite{Wang13} holds for $s>\frac12$, and the norm inflation of Shen--Wang \cite{ShenWang} holds for
$s\in(-\infty,0)\cup(0,\frac12]$. At $s=0$, mass conservation precludes norm inflation, although Wang's failure of $C^3$ regularity of the solution map still applies. In the nonsquare case, Theorem~\ref{thm:gwp} gives small-mass global well-posedness for every $s>0$. Thus, for $0<s\le\frac12$, small-mass global well-posedness in the nonsquare case contrasts with norm inflation in the square case.

These conclusions yield a dichotomy on two dense classes of rectangular tori.
\begin{corollary}[Dense arithmetic dichotomy]\label{cor:domain}
Fix $0<s\le\frac12$. The cubic hyperbolic equation \eqref{equ:HNLS} on $\T^{2}_{\rho,1}$ is globally well-posed in $H^{s}$ for data with sufficiently small $L^{2}$-norm whenever $\rho^{2}\in\Q$ and $\rho\notin\Q$, whereas it exhibits norm inflation at the origin in $H^{s}$ whenever $\rho\in\Q_{>0}$.
Both classes of aspect ratios are dense in $(0,\infty)$. Consequently, every neighborhood of every positive aspect ratio contains tori exhibiting each of these two behaviors.
\end{corollary}
\begin{proof}
The global well-posedness assertion follows from Corollary~\ref{cor:quadratic,irrational,tori}. If $\rho=m/n$ with coprime $m,n\in\N$, then \eqref{equ:rect} gives the normalized coefficients
$a=n^{2}$ and $b=-m^{2}$, so $-ab=(mn)^{2}$ is a square.
The square-case reduction above therefore transfers the norm inflation of Shen--Wang \cite{ShenWang} to $\T^{2}_{\rho,1}$. Finally, $\Q_{>0}$ is dense in $(0,\infty)$, as is $\sqrt2\,\Q_{>0}$. Every $\rho\in\sqrt2\,\Q_{>0}$ satisfies $\rho^{2}\in\Q$ and $\rho\notin\Q$, which proves the claimed
density of the quadratic irrational aspect ratios.
\end{proof}

This gives arithmetic sensitivity of the well-posedness threshold
throughout the space of rectangular aspect ratios.
\subsubsection{Comparison with long-time Strichartz estimates}\label{section:Comparison}
The tori in Corollary~\ref{cor:quadratic,irrational,tori} have irrational aspect ratios but rational squared aspect ratios, so the ratio of their dispersion coefficients is rational. The sharp logarithmic loss in
Theorem~\ref{thm:sharp,Strichartz,L4} is governed by exact lattice arithmetic and, in the hyperbolic case, the absence of rational null directions.

The long-time refinements of Deng--Germain--Guth \cite{DGG17} exploit Diophantine properties of generic rectangular metrics. The subsequent work of Deng--Germain--Guth--Rydin Myerson \cite{DGGRM22} treats
generic non-rectangular two-dimensional tori and also discusses extensions to hyperbolic operators \cite[Remark 1.3]{DGGRM22}.

For a quantitative comparison with the decoupling argument of Guth--Maldague--Oh \cite[Section 4.2]{GMO26}, we use the normalization
\begin{align*}
\widehat{S_{\alpha}(t)f}(k)=e^{-it(k_{1}^{2}+\alpha k_{2}^{2})}\wh f(k),\quad \alpha<0.
\end{align*}
The fixed normalization of time in \cite{GMO26} does not affect the comparison. In our hyperbolic setting, \eqref{equ:rect} gives $\alpha=-p/q=-\rho^{2}$. Thus the tori in Corollary~\ref{cor:quadratic,irrational,tori} correspond to $\alpha\in\Q_{<0}$ with $\sqrt{-\alpha}\notin\Q$.

Guth--Maldague--Oh establish an $\ell^{2}$ decoupling theorem for the
hyperbolic paraboloid using overlapping flat rectangles. Their
Section 4.2 applies this theorem to
\begin{align*}
\Phi(\xi)=\xi_{1}^{2}-\xi_{2}^{2},\quad
\Lambda_\delta=\bigl(\delta\Z\times\sqrt2\,\delta\Z\bigr)\cap[0,1]^{2}.
\end{align*}
Since $\Phi(m\delta,\sqrt2\,n\delta)=\delta^{2}(m^{2}-2n^{2})$,
the lattice-spacing parameter $\sqrt2$ in this example corresponds
to the rational dispersion coefficient $\alpha=-2$.
Their argument yields
\begin{align*}
    \n{S_{-2}(t)P_{\le N}f}_{L^{4}([0,T]\times\T^{2})}
    \lesssim_\varepsilon
    N^{\ve} T^{1/4}\n{f}_{L^{2}},
    \quad T\ge N,
\end{align*}
for every $\varepsilon>0$.
The same argument extends to $\alpha=-p/q$, where
$p,q\in\N$ are coprime and $p/q$ is not a square in $\Q$.
Indeed, replace the lattice spacing $\sqrt2$ by $\sqrt{p/q}$.
The flat-piece counting step in \cite[Section 4.2]{GMO26} supplies
the separation required in their decoupling argument. As a consequence
of this adaptation, one obtains
\begin{equation}\label{equ:compare,gmo}
\n{S_{\alpha}(t)P_{\le N}f}_{L^{4}([0,T]\times\T^{2})}\lesssim_{\varepsilon,\alpha}N^{\ve} T^{1/4}\n{f}_{L^{2}},
\quad T\ge N,
\end{equation}
throughout this nonsquare rational-coefficient family.

Our results determine the precise endpoint behavior. Since $S_{\alpha}(t)=S_{q,-p}(t/q)$, Theorem~\ref{thm:sharp,short-time,L4} gives the lossless estimate
\begin{equation} \label{equ:compare,lossless}
\n{S_{\alpha}(t)P_{\le N}f}_{L^{4}(I\times\T^{2})} \lesssim_{\alpha}\n{f}_{L^{2}},
\quad |I|\le c_{\alpha}(\log N)^{-1},
\end{equation}
whose time scale is optimal up to a constant. This does not follow merely by restricting the fixed-period estimate of
Theorem~\ref{thm:sharp,Strichartz,L4}, since restriction alone retains the logarithmic loss.
Partitioning $[0,T]$ into intervals of the length in
\eqref{equ:compare,lossless} and summing fourth powers yields
\begin{equation}\label{equ:compare,log}
\n{S_{\alpha}(t)P_{\le N}f}_{L^{4}([0,T]\times\T^{2})}\lesssim_{\alpha}(T\log N)^{1/4}\n{f}_{L^{2}},
\quad T\ge(\log N)^{-1}.
\end{equation}
Compared with \eqref{equ:compare,gmo}, this replaces the
$N^{\ve}$ loss by the optimal factor $(\log N)^{1/4}$
and extends the corresponding time-dependent estimate from $T\ge N$
to $T\ge(\log N)^{-1}$. At the latter threshold, the estimate is lossless.

\subsection{Organization of the paper}
Section~\ref{section:Preliminaries} introduces the Fourier and function-space
notation and records the resonance identities and null-direction
criterion. Section~\ref{section:Sharp Strichartz} develops the counting estimates and
proves Theorems~\ref{thm:sharp,Strichartz,L4} and~\ref{thm:sharp,short-time,L4}, including the
optimality of the frequency losses and time scales.
Section~\ref{section:Global well-posedness} uses the lossless short-time estimate to establish
the small-mass global well-posedness result in Theorem~\ref{thm:gwp}.
Appendix~\ref{appendix:Typed-rectangle} supplies the arithmetic reduction and
type-by-type counting arguments adapted from Herr--Kwak \cite{HK24}
that underpin Section~\ref{section:Sharp Strichartz}.

\section{Preliminaries}\label{section:Preliminaries}
This section collects the notation and auxiliary material used throughout
the paper.  We first fix the Fourier and function-space conventions,
including the adapted $Y^{s}$ spaces used in the nonlinear argument.  We
then record the resonance identities and the basic null-direction
criterion.

\subsection{Fourier and frequency notation}
We use the Fourier convention
\begin{align*}
f(x)=\sum_{k\in\Z^{2}}\wh f(k)e^{ik\cdot x},\quad \wh{f}(k)=\frac1{(2\pi)^{2}}\int_{\T^{2}}f(x)
e^{-ik\cdot x}\,dx.
\end{align*}
For a set $S\subset\Z^{2}$, let $P_S$ denote the sharp Fourier projection
\begin{align*}
        \widehat{P_Sf}(k)=\bfone_S(k)\wh f(k).
\end{align*}
For $N\ge1$, we set
\begin{align*}
 A_{N}=[-N,N]^{2}\cap\Z^{2}, \quad P_{\le N}=P_{A_{N}}.
\end{align*}
Dyadic frequency scales range over $\{2^j:j\in\Z_{\ge 0}\}$.
With the convention $P_{\le1/2}=0$, for dyadic $M\ge1$ we write
\begin{align*}
&P_{M}=P_{\le M}-P_{\le M/2},\quad
P_{<M}=P_{\le M/2},\quad
P_{\ge M}=1-P_{<M},\\
&u_{M}=P_{M}u,\quad u_{\ge M}=P_{\ge M}u.
\end{align*}
In particular, $P_{1}=P_{\le1}$. We use $\langle\xi\rangle=(1+|\xi|^{2})^{1/2}$.
%All $L^{p}$-norms on $\T^{2}$ below are taken with respect to Lebesgue
%measure $dx$, so Parseval's identity reads $\n f_{L^{2}}=2\pi\n{\wh f}_{\ell^{2}}$.
%Harmless powers of $2\pi$ are absorbed into implicit constants.

\subsection{$U$-$V$ spaces}\label{section:Y spaces}
We recall the $U^{p}/V^{p}$ function-space framework from
\cite[Section 2]{HTT11} and the associated $Y^{s}$ spaces in the form
used in \cite[Section 4]{HK24}.

Let $H$ be a Hilbert space and $1\le p<\infty$. A $U^{p}(\R;H)$-atom is a step function
\begin{align*}
 A(t)=\sum_{k=1}^{K}\bfone_{[t_{k-1},t_k)}(t)\,h_{k-1},
 \quad \sum_{k=0}^{K-1}\n{h_k}_H^{p}=1,
\end{align*}
where $-\infty=t_0<t_{1}<\cdots<t_K=\infty$,
$h_k\in H$, and $h_0=0$. 
The space $U^p(\R;H)$ consists of sums
$u=\sum_n c_n A_n$, where $A_n$ are $U^p$-atoms and
$\sum_n|c_n|<\infty$, equipped with the norm
\begin{align*}
\n{u}_{U^p(\R;H)}
=\inf\bigg\{
\sum_n|c_n|:
u=\sum_n c_n A_n,\quad A_n\text{ are }U^p\text{-atoms}
\bigg\}.
\end{align*}
The space $V^{p}(\R;H)$ consists of right-continuous functions
with limits at $\pm\infty$, vanishing at $-\infty$, and finite
\begin{align*}
 \n{v}_{V^{p}(\R;H)}=\sup_{\{t_k\}}\bigg(\sum_k\n{v(t_k)-v(t_{k-1})}_H^{p}\bigg)^{1/p},
\end{align*}
where the supremum is taken over finite partitions
$-\infty=t_{0}<t_{1}<\cdots<t_K=\infty$.  At the terminal endpoint in the
variation sum, we use the standard convention $v(t_K)=0$, independently
of the actual limit as $t\to+\infty$. At $t_{0}=-\infty$, we use the
actual limit, which is zero. Thus, $V^{p}$ is the space denoted by $V^{p}_{\mathrm{rc}}$ in
\cite[Definition 2.2]{HTT11}. The adapted spaces $U_\Omega^{p}L^{2}$ and $V_\Omega^{p}L^{2}$ with $\Omega$ as in \eqref{equ:disp} are defined by applying $U^{p}(\R;L^{2})$ and $V^{p}(\R;L^{2})$ to $S(-t)u(t)$.

For $s\in\R$, we define $Y_\omega^{s}$ by
\begin{align*}
\n{u}_{Y_\omega^{s}} =\bigg(\sum_{\xi\in\Z^{2}}\langle\xi\rangle^{2s}
\n{e^{it\omega(\xi)}\wh u(t,\xi)}_{V^{2}_t}^{2}\bigg)^{1/2}.
\end{align*}
For an interval $I$, the corresponding restriction norm is
\begin{align*}
\n u_{Y_\omega^{s}(I)}=\inf\bigl\{\n U_{Y_\omega^{s}}:U|_I=u\bigr\}.
\end{align*}
We suppress the subscript $\omega$ below and write $Y^{s}$.

We recall the following standard properties of the restriction
spaces; see \cite[Proposition 4.2]{HK24} or \cite[Section 2]{HTT11}.
\begin{lemma}[Basic properties of $Y^{s}$]\label{lemma:Y space}
Let $s\in\R$ and let $I\subset\R$ be a bounded interval. Then we have
\begin{enumerate}
\item if $A,B\subset\Z^{2}$ are disjoint, then
\begin{align*}
\n{P_{A\cup B}u}_{Y^{s}(I)}^{2}=\n{P_Au}_{Y^{s}(I)}^{2}+\n{P_Bu}_{Y^{s}(I)}^{2};
\end{align*}
\item for $\phi \in H^{s} (\T^{2})$, we have
\begin{align*}
\n{S(t)\phi}_{Y^{s}(I)}\lesssim\n{\phi}_{H^{s}};
\end{align*}
\item for $u \in Y^{s} (I)$, we have
\begin{align*}
\n{u}_{L_t^\infty H_x^{s}(I\times\T^{2})}\lesssim
\n{u}_{Y^{s}(I)};
\end{align*}
\item if $t_{0}\in I$ and $F\in L^{1}(I;H^{s}(\T^{2}))$, then for the
Duhamel operator
\begin{align*}
\mathcal I_{t_{0}}(F)(t)=\int_{t_{0}}^t S(t-t')F(t')\,dt',
\end{align*}
one has the dual estimate
\begin{align*}
\n{\mathcal I_{t_{0}}(F)}_{Y^{s}(I)}\lesssim
\sup_{\n{v}_{Y^{-s}(I)}\le1}
\bbabs{\int_I\int_{\T^{2}}F\overline v\,dxdt}.
\end{align*} 
\end{enumerate}
\end{lemma}

\subsection{Resonance notation and null directions}
For $u,v\in\Z^{2}$, we define the symmetric bilinear form
\begin{align*}
 B_{a,b}(u,v)=a u_{1}v_{1}+b u_{2}v_{2}.
\end{align*}
A lattice parallelogram is an ordered quadruple
$Q=(\xi_{1},\xi_{2},\xi_{3},\xi_{4})\in(\Z^{2})^{4}$ satisfying
\begin{align}\label{equ:para}
\xi_{1}+\xi_{3}=\xi_{2}+\xi_{4},
\end{align}
and we allow repeated vertices and collinear configurations.
For every such quadruple, we have
\begin{equation}\label{equ:res}
 \omega(\xi_{1})-\omega(\xi_{2})+\omega(\xi_{3})-\omega(\xi_{4})=
 2B_{a,b}(\xi_{1}-\xi_{2},\xi_{1}-\xi_{4}).
\end{equation}
Let $\cQ$ denote the set of all lattice parallelograms, and for
$\tau\in\Z_{\ge0}$, we define
\begin{equation}\label{equ:Qtau}
 \cQ_\tau^{a,b} =
\lrs{Q\in\cQ:\abs{\omega(\xi_{1})-\omega(\xi_{2})+\omega(\xi_{3})-\omega(\xi_{4})}=\tau}.
\end{equation}
We call a lattice parallelogram
$Q=(\xi_{1},\xi_{2},\xi_{3},\xi_{4})\in\cQ$ a
$B_{a,b}$-rectangle, if its adjacent side vectors are $B_{a,b}$-orthogonal, that is,
\begin{align*}
 B_{a,b}(\xi_{1}-\xi_{2},\xi_{1}-\xi_{4})=0.
\end{align*}
By \eqref{equ:res}, the $B_{a,b}$-rectangles are precisely the elements of $\cQ_{0}^{a,b}$.

The square case can now be characterized by the existence of a lattice
null direction.

\begin{lemma}[Null directions]
\label{lemma:null}
Let $a,b\in\Z\setminus\{0\}$. Then the following are equivalent:
\begin{enumerate}
\item the equation $a n^{2}+b m^{2}=0$ has a nonzero solution $(n,m)\in\Z^{2}$.
\item $ab<0$ and $-ab$ is a square.
\end{enumerate}
\end{lemma}

\begin{proof}
If $an^{2}+bm^{2}=0$ has a nonzero integer solution, then $nm\ne0$,
$ab<0$, and
\begin{align*}
        -ab=\left(\frac{an}{m}\right)^{2}.
\end{align*}
An integer that is a square in $\Q$ is a square in $\Z$.
Conversely, if $h^{2}=-ab$ for some $h\in\Z$, then
$(n,m)=(h,a)$ is a nonzero integer solution, since
$ah^{2}+ba^{2}=0$.
\end{proof}

\section{Sharp Strichartz estimates}\label{section:Sharp Strichartz}
The purpose of this section is to prove the Strichartz estimates,
Theorems~\ref{thm:sharp,Strichartz,L4} and~\ref{thm:sharp,short-time,L4}. 

\subsection{Counting input for the Strichartz estimates}\label{section:Counting input}
For a function $f:\Z^{2}\to[0,\infty)$ and a parallelogram
$Q=(\xi_{1},\xi_{2},\xi_{3},\xi_{4})$, we write
\begin{equation}\label{equ:GQ}
f(Q)=f(\xi_{1})f(\xi_{2})f(\xi_{3})f(\xi_{4}).
\end{equation}
We first recall the oscillatory reduction in the following form.

\begin{lemma}[Oscillatory reduction, {\cite[Proof of Theorem 1.2]{HK24}}]\label{lemma:Oscillatory reduction}
Let $T>0$, and assume $\wh g\ge0$ is finitely supported. Then
\begin{align}\label{equ:osc}
\n{S(t)g}_{L^{4}([0,T]\times\T^{2})}^{4} &\lesssim T\sum_{Q\in\cQ_{0}^{a,b}}\wh g(Q) +
\sum_{\tau>0}\min\left\{T,\frac1{T\tau^{2}}\right\}
\sum_{Q\in\cQ_\tau^{a,b}}\wh g(Q),
\end{align}
where $S(t) = S_{a,b} (t)$ and $\cQ_\tau^{a,b}$ are defined in \eqref{equ:flow} and \eqref{equ:Qtau}, respectively; and
$\wh g(Q)$ follows the convention \eqref{equ:GQ}.
\end{lemma}
%\begin{proof}
%Following \cite[Proof of Theorem 1.2]{HK24}, we majorize
%$\bfone_{[0,T]}(t)$ by
%$(2-t/T)\bfone_{[0,2T]}(t)$. After expanding the fourth power,
%spatial orthogonality imposes the parallelogram relation. For the signed
%modulation $\sigma$, the real part of the time integral is
%\begin{align*}
%\operatorname{Re}\int_{0}^{2T}(2-t/T)e^{-it\sigma}\,dt=
%\begin{cases}
%2T, & \sigma=0,\\
%\frac{1-\cos(2T\sigma)}{T\sigma^{2}}, & \sigma\ne0.
%\end{cases}
%\end{align*}
%
%This kernel is nonnegative and, for $\sigma\ne0$, bounded by
%$C\min\{T,(T\sigma^{2})^{-1}\}$. Since $\wh g(Q)\ge0$, grouping
%the terms according to $\tau=|\sigma|$ proves \eqref{equ:osc}.
%\end{proof}

\subsubsection{The arithmetic incidence mechanism}
We now introduce the geometric ingredients.

For $\xi:=(\xi_{x_{1}},\xi_{x_{2}})\in\Z^{2}\setminus\{0\}$, we set
\begin{align*}
 d_{a,b}(\xi)=\gcd(|a\xi_{x_{1}}|,|b\xi_{x_{2}}|),\quad
\xi^{\perp_{a,b}}=\frac1{d_{a,b}(\xi)}(-b\xi_{x_{2}},a\xi_{x_{1}}).
\end{align*}

We record precisely the lattice facts that replace Euclidean
orthogonality.
\begin{lemma}[Integral $B_{a,b}$-orthogonality]\label{lemma:Borth}
Let $\xi\in\Z^{2}\setminus\{0\}$. Then it holds that
\begin{enumerate}
\item $\xi^{\perp_{a,b}}$ is a primitive integer vector and
\begin{align*}
        \{\eta\in\Z^{2}:B_{a,b}(\eta,\xi)=0\}
        =\Z\xi^{\perp_{a,b}}.
\end{align*}
\item The homomorphism $z\mapsto B_{a,b}(z,\xi)$ from $\Z^{2}$ to
$\Z$ has image
\begin{align*}
B_{a,b}(\Z^{2},\xi)=d_{a,b}(\xi)\Z.
\end{align*}
Consequently, for $c\in\Z$, the set
\begin{align*}
\{\eta\in\Z^{2}:B_{a,b}(\eta,\xi)=c\}
\end{align*}
is either empty or a translate of $\Z\xi^{\perp_{a,b}}$.
\item If $-ab$ is not a square, then
$B_{a,b}(\xi,\xi)\ne0$, and hence $\xi$ and
$\xi^{\perp_{a,b}}$ are linearly independent.
\end{enumerate}
\end{lemma}
\begin{proof}
We put $d=d_{a,b}(\xi)$. The two coordinates of
$\xi^{\perp_{a,b}}=(-b\xi_{x_{2}}/d,a\xi_{x_{1}}/d)$ are coprime, so this vector
is primitive, and it is plainly $B_{a,b}$-orthogonal to $\xi$.
Conversely, if $\xi_{x_{1}}=0$, orthogonality is equivalent to
$\eta_{x_{2}}=0$, and $\xi^{\perp_{a,b}}$ is a signed coordinate
unit vector spanning this kernel. The case $\xi_{x_{2}}=0$ is identical.
We may therefore assume that both coordinates of $\xi$ are nonzero. If
\begin{align}\label{equ:Bab=0}
 B_{a,b}(\xi,\eta)=a\xi_{x_{1}}\eta_{x_{1}}+b\xi_{x_{2}}\eta_{x_{2}}=0,
\end{align}
the coprimality of
$a\xi_{x_{1}}/d$ and $b\xi_{x_{2}}/d$ implies $\eta$ must be of the form
$\eta=(k_{1}\frac{b\xi_{x_{2}}}{d},k_{2}\frac{a\xi_{x_{1}}}{d})$ for some integers $k_{1}$ and $k_{2}$. Moreover, the equality \eqref{equ:Bab=0} forces $k_{1}=-k_{2}$, so
$\eta\in\Z\xi^{\perp_{a,b}}$. This proves the first assertion.

For the second assertion, we notice that the image of $z\mapsto B_{a,b}(z,\xi)$ is the ideal generated by
$a\xi_{x_{1}}$ and $b\xi_{x_{2}}$, which is
$\gcd(|a\xi_{x_{1}}|,|b\xi_{x_{2}}|)\Z=d\Z$.
If an affine level is nonempty, subtracting one of its points reduces
it to the kernel described in the first assertion.

Finally, the last assertion that $B_{a,b}(\xi,\xi)\neq 0$ follows from  Lemma \ref{lemma:null}. Moreover,
linear dependence of $\xi$ and $\xi^{\perp_{a,b}}$, together with their
orthogonality, would give $B_{a,b}(\xi,\xi)=0$.
\end{proof}

Thus, once one side direction of a $B_{a,b}$-rectangle is fixed, the
adjacent side is constrained to lie along a unique primitive lattice direction.
Throughout the remainder of Subsection~\ref{section:Counting input}, we assume that
$-ab$ is not a square.

\begin{remark}\rm
A $B_{a,b}$-rectangle need not be a Euclidean rectangle. The extension
of the Herr--Kwak argument uses three facts following from
Lemma~\ref{lemma:Borth} and the resonance identity \eqref{equ:res}:
the orthogonal lattice is $\Z\xi^{\perp_{a,b}}$, the two side
directions are independent in the nonsquare case, and a parallelogram
of modulation $\tau$ with nonzero side $\xi$ satisfies
$2d_{a,b}(\xi)\mid\tau$. 
\end{remark}

An $(a,b)$-cross is a triple $(p,\ell_{1},\ell_{2})$, where $p\in\Z^{2}$, $\ell_{1},\ell_{2}\subset\R^{2}$ are affine lines through $p$, and their direction vectors are $B_{a,b}$-orthogonal. If
\begin{align*}
        Q=(\xi_{1},\xi_{2},\xi_{3},\xi_{4})\in\cQ_{0}^{a,b}
\end{align*}
has four distinct vertices, then each vertex determines an $(a,b)$-cross.

Let $S\subset\Z^{2}$ be finite. For a line $\ell\subset\R^{2}$, we write
\begin{align}\label{equ:rS,number}
r_S(\ell)=\#(\ell\cap S).
\end{align}
The Szemer\'edi--Trotter theorem \cite[Corollary 8.5]{tao2006additive} gives, for $K\ge2$,
\begin{equation}\label{equ:rich}
\#\{\ell:r_S(\ell)\ge K\}\lesssim\frac{(\#S)^{2}}{K^3}+\frac{\#S}{K}.
\end{equation}
The implicit constant is absolute. Suppose $S_j \subset \Z^{2}$ is a dyadic level set with $\#S_j\le2^j$.
We follow the notation from \cite{HK24} in the following.
For a vertex $\xi\in S_j$, we classify an $(a,b)$-cross
$(\xi,\ell_{1},\ell_{2})$ by
\begin{align*}
R_j(\xi;\ell_{1},\ell_{2})=\max\{r_{S_j}(\ell_{1}),r_{S_j}(\ell_{2})\}.
\end{align*}
Fix a sufficiently large absolute constant $C$. The cross is
\begin{itemize}
\item of
\emph{type $1$} if $R_j\ge2^{j/2+C}$,
\item of \emph{type $2$} if $2\le R_j<2^{j/2+C}$,
and
\item
of \emph{type $3$} if $R_j=1$.
\end{itemize}
We also call these crosses rich,
medium, and poor, respectively. A vertex of a rectangle is of type
$\alpha$ when its two adjacent side lines form a cross of type
$\alpha$, with richness measured in the level set containing that
vertex.

%{\color{blue}In the argument used below, all points incident to two distinct
%rich lines are removed from the current level set. Consequently, every
%remaining point lies on at most one rich line, which is precisely the
%hypothesis in Proposition \ref{prop:count}. Type $1$ crosses
%are then handled by uniqueness from an opposite pair of vertices; type
%$2$ crosses produce the single logarithmic summation; and type $3$
%crosses are controlled by a one-corner count.}

Lemmas~\ref{lemma:corner}--\ref{lemma:divisor side} below are the
$B_{a,b}$-orthogonal analogues of Herr--Kwak's counting lemmas
\cite[Lemmas 3.4--3.5]{HK24}, with their Lemma 3.5 separated into
unweighted and divisor-weighted estimates. They provide the counting
input for Proposition~\ref{prop:count} and the subsequent Strichartz
estimates. We include brief proofs to explain how the arguments adapt
to the nonsquare setting, using the orthogonal lattice structure
established in Lemma~\ref{lemma:Borth} and the modified divisor
$d_{a,b}$.

\begin{lemma}[A one-corner count]\label{lemma:corner}
Let $S_{j_{1}},S_{j_{2}},S_{j_{3}},S_{j_{4}}\subset\Z^{2}$ satisfy $\#S_{j_i}\le2^{j_i}$.
Let $a,b\in \Z$ such that $-ab$ is not a square.
For a fixed integer $a_{3}\ge0$, the number of $B_{a,b}$-rectangles
\begin{align*}
Q=(\xi_{1},\xi_{2},\xi_{3},\xi_{4})\in\cQ_{0}^{a,b}\cap(S_{j_{1}}\times S_{j_{2}}\times S_{j_{3}}\times S_{j_{4}})
\end{align*}
with four distinct vertices such that
\begin{align}\label{equ:condition,corner}
 \#\left(\overleftrightarrow{\xi_{2}\xi_{3}}\cap S_{j_{3}}\right)<2^{a_{3}+1}
\end{align}
is bounded by  $O(2^{j_{1}+j_{2}+a_{3}})$.
\end{lemma}
\begin{proof}
There are $O(2^{j_1+j_2})$ choices for $(\xi_1,\xi_2)$. Once this pair is fixed, the condition $Q\in\cQ_0^{a,b}$ forces $\xi_3$ to lie on the unique affine line through $\xi_2$ whose direction is $B_{a,b}$-orthogonal to $\xi_1-\xi_2$. The hypothesis gives $O(2^{a_3})$ possible choices for $\xi_3$. Finally $\xi_4$ is determined by the parallelogram relation.
\end{proof}

\begin{lemma}[Medium-line count]\label{lemma:Medium-line}
Let $a_{1},a_{2},a_{3},a_{4},j_{1},j_{2},j_{3},j_{4}$ be nonnegative integers, and let
$S_{j_i}\subset\Z^{2}$ satisfy $\#S_{j_i}\le2^{j_i}$.
Assume that $-ab$ is not a square and
\begin{align*}
        1\le a_{1}<j_{1}/2+C.
\end{align*}
Let $\cQ_{0}^{a,b}(\vec j,\vec a)$ be the set of $B_{a,b}$-rectangles
with four distinct vertices, $\xi_i\in S_{j_i}$, such that, with
cyclic indices,
\begin{equation}\label{equ:shell}
2^{a_i}\le\max\lr{
r_{S_{j_i}}(\overleftrightarrow{\xi_i\xi_{i-1}}),
r_{S_{j_i}}(\overleftrightarrow{\xi_i\xi_{i+1}})
}<2^{a_i+1}
\quad(i=1,2,3,4),
\end{equation}
where $r_{S}(l)$ is defined by \eqref{equ:rS,number}.
Then we have
\begin{align}\label{equ:Medium,count,a4}
 \#\cQ_{0}^{a,b}(\vec j,\vec a) \lesssim
        2^{2j_{1}-2a_{1}+a_{2}+a_{4}}
\end{align}
and also
\begin{align}\label{equ:Medium,count,a3}
 \#\cQ_{0}^{a,b}(\vec j,\vec a) \lesssim
        2^{2j_{1}-2a_{1}+a_{2}+a_{3}}.
\end{align}
\end{lemma}

\begin{proof}
By Szemer\'edi--Trotter theorem in \eqref{equ:rich}, the number of possible lines at $\xi_1$,
which meet $S_{j_1}$ in between $2^{a_1}$ and $2^{a_1+1}$ points,
is
\[
        O(2^{2j_1-3a_1}+2^{j_1-a_1})
        =
        O_C(2^{2j_1-3a_1}),
\]
where we used $a_1<j_1/2+C$. Each such line contains
$O(2^{a_1})$ points of $S_{j_1}$. After choosing $\xi_1$, its
other side line is the unique line through $\xi_1$ in the
$B_{a,b}$-orthogonal direction. Hence there are
$O_C(2^{2j_1-2a_1})$ possible crosses at $\xi_1$.

Once the cross is fixed, the upper bounds in
\eqref{equ:shell} give $O(2^{a_2})$ choices for $\xi_2$
on its side line and $O(2^{a_4})$ choices for $\xi_4$ on the
other side line; $\xi_3$ is then determined by the parallelogram
relation. This proves the first estimate \eqref{equ:Medium,count,a4}. Alternatively, after choosing
$\xi_2$, choose $\xi_3$ on the line through $\xi_2$ parallel
to the second side. The upper bound at $\xi_3$ gives
$O(2^{a_3})$ choices, and $\xi_4$ is then determined. This proves
the second estimate \eqref{equ:Medium,count,a3}.  
\end{proof}

\begin{lemma}[Divisor-weighted side count]\label{lemma:divisor side}
With the notation of Lemma \ref{lemma:Medium-line},
we have
\begin{align*}
\sum_{Q\in\cQ_{0}^{a,b}(\vec j,\vec a)}\frac1{d_{a,b}(\xi_{1}-\xi_{4})}\lesssim2^{2j_{1}-2a_{1}+a_{2}+a_{4}/2}.
\end{align*}
\end{lemma}

\begin{proof}
As in the proof of Lemma~\ref{lemma:Medium-line}, there are
$O_C(2^{2j_{1}-2a_{1}})$ choices of a cross
$(\xi_{1},l_{12},l_{14})$. For each such cross, the condition
\eqref{equ:shell} at $\xi_{2}$ gives $O(2^{a_{2}})$ admissible choices
of $\xi_{2}\in l_{12}\cap S_{j_{2}}$. Thus, there are at most
$O_C(2^{2j_{1}-2a_{1}+a_{2}})$ choices of the cross together with
$\xi_{2}$.

Fix one such choice. The remaining choice is a point
$\xi_{4}\in l_{14}\cap S_{j_{4}}$, since the parallelogram relation
\eqref{equ:para} uniquely determines
\begin{align*}
        \xi_{3}=\xi_{2}+\xi_{4}-\xi_{1}.
\end{align*}
Hence each admissible $\xi_{4}$ determines exactly one rectangle $Q$.
If there are no admissible choices, their contribution is zero.
Otherwise, by \eqref{equ:shell} at $\xi_{4}$, their number $M$ satisfies
$1\le M<2^{a_{4}+1}$. Choose a primitive integer vector $v$ parallel
to $l_{14}$ and write these points as
\begin{align*}
        \xi_{4}^{(r)}=\xi_{1}+n_r v,\quad 1\le r\le M,
\end{align*}
where the integers $n_r\ne0$ are distinct. Relabel them so that
$|n_1|\le\cdots\le|n_M|$. Since each positive integer occurs at most
twice among the absolute values $|n_r|$, we have $|n_r|\ge r/2$.
Moreover,
\begin{align*}
        d_{a,b}(\xi_{1}-\xi_{4}^{(r)})
        =|n_r|\,d_{a,b}(v),\quad d_{a,b}(v)\ge1.
\end{align*}
Consequently, the total divisor weight for this fixed cross and
$\xi_{2}$ is bounded by
\begin{align*}
 \sum_{r=1}^{M}\frac1{d_{a,b}(\xi_{1}-\xi_{4}^{(r)})}
 &\le \frac{2}{d_{a,b}(v)}\sum_{r=1}^{M}\frac1r
 \lesssim a_{4}+1
 \lesssim 2^{a_{4}/2}.
\end{align*}
Summing over the $O_C(2^{2j_{1}-2a_{1}+a_{2}})$ choices of the
cross together with $\xi_{2}$ gives
\begin{align*}
 \sum_{Q\in\cQ_{0}^{a,b}(\vec j,\vec a)}
 \frac1{d_{a,b}(\xi_{1}-\xi_{4})}
 \lesssim_C
 2^{2j_{1}-2a_{1}+a_{2}}\,2^{a_{4}/2},
\end{align*}
as required. 
\end{proof}

\begin{remark}
\rm 
Lemmas 3.4–3.6 show that, in the nonsquare case, i.e. $-ab$ is a non-square, the basic counting structure of Herr–Kwak is preserved when Euclidean rectangles are replaced by $B_{a,b}$-rectangles. With these ingredients in place, the reduction of the Strichartz estimates to counting problems and the subsequent summation arguments follow the same scheme. For completeness, we detail the full counting argument and the resulting Strichartz estimates in the following subsections, and present the reduction to typed resonant rectangles and the associated estimates in Appendix A.
\end{remark}

\subsubsection{The full counting argument}

The following is one of the main ingredients for the proof of Theorem \ref{thm:sharp,short-time,L4}.
It is an analogue of \cite[Proposition 3.1]{HK24}.

\begin{proposition}[Dyadic counting proposition]\label{prop:count}
Let
\begin{align*}
f=\sum_{j=0}^m\lambda_j2^{-j/2}\bfone_{S_j},
\end{align*}
where $m\ge1$, the sets $S_{0},\ldots,S_m\subset\Z^{2}$ are pairwise
disjoint, $\#S_j\le2^j$, and $\lambda_j\ge0$. Assume that for every
$\xi\in S_j$, there exists at most one line $\ell\ni\xi$ such that
\begin{align*}
\#(\ell\cap S_j)\ge2^{j/2+C}.
\end{align*}
If $-ab$ is not a square, then
\begin{equation}\label{equ:dy0}
\sum_{Q\in\cQ_{0}^{a,b}} f(Q)\lesssimab m\n{\lambda}_{\ell^{2}}^{4}
\end{equation}
and
\begin{equation}\label{equ:dyM}
\sup_{M\in2^\N}\frac1M\sum_{\tau\sim M}\sum_{Q\in\cQ_\tau^{a,b}}f(Q)\lesssimab
 \n{\lambda}_{\ell^{2}}^{4}.
\end{equation}
Here, $\tau\sim M$ means $M\le\tau<2M$.
\end{proposition}
\begin{proof}
For $\tau=0$, Lemmas~\ref{lemma:type,red} and~\ref{lemma:type,est},
proved in Appendix~\ref{appendix:Typed-rectangle}, give
\begin{align*}
 \sum_{Q\in\cQ_{0}^{a,b}}f(Q) \lesssimab& \max_{\alpha,\beta=1,2,3}
\sum_{\substack{Q\in\cQ_{0,\alpha,\beta}^{a,b}}}f(Q) + \n f_{\ell^{2}(\Z^{2})}^{4}\\
\lesssim& m\n{\lambda}_{\ell^{2}}^{4}+\n f_{\ell^{2}}^{4}
        \lesssim m\n{\lambda}_{\ell^{2}}^{4},
\end{align*}
where in the last inequality we have used that $\n f_{\ell^{2}}\le\n{\lambda}_{\ell^{2}}$ and $m\ge1$.

For \eqref{equ:dyM}, using Lemma \ref{lemma:type,red} again, we have
\begin{align}
\frac{1}{M}\sum_{\tau\sim M}\sum_{Q\in\cQ_\tau^{a,b}}f(Q)
\lesssimab&
\frac{1}{M}\sum_{\tau\sim M} \lrc{\max_{\alpha,\beta=1,2,3}
\sum_{\substack{Q\in\cQ_{0,\alpha,\beta}^{a,b}\notag\\
d_{a,b}(\xi_{1}-\xi_{4})\mid \tau}}f(Q) + \n f_{\ell^{2}(\Z^{2})}^{4}}\\
\lesssim& \max_{\alpha,\beta=1,2,3}\frac{1}{M}\sum_{\tau\sim M}
\sum_{\substack{Q\in\cQ_{0,\alpha,\beta}^{a,b}\notag\\
d_{a,b}(\xi_{1}-\xi_{4})\mid \tau}}f(Q) + \n f_{\ell^{2}(\Z^{2})}^{4}\notag\\
\lesssim& \max_{\al,\be=1,2,3}\sum_{Q\in Q_{0,\al,\be}^{a,b}}\frac{1}{d_{a,b}(\xi_{1}-\xi_{4})}f(Q)+
\n f_{\ell^{2}(\Z^{2})}^{4},\label{equ:tau,M,dab}
\end{align}
where in the last inequality we have used the following estimate
\begin{align*}
\frac{1}{M}\#\lr{\tau\in\Z:\tau\sim M,\ d\mid\tau}
\leq \frac{1}{M}\lr{\frac{M}{d}+1}\lesssim \frac{1}{d},
\end{align*}
provided $M \gtrsim d$.

For type pairs $(\al,\be)\neq (2,2)$, we substitute $d_{a,b}\ge1$
together with \eqref{equ:type,off} into \eqref{equ:tau,M,dab}. In the case $(\al,\be)=(2,2)$, we substitute the weighted estimate \eqref{equ:type22-div} into \eqref{equ:tau,M,dab}. Therefore, we complete the proof of \eqref{equ:dyM}.
\end{proof}

\Needspace{6\baselineskip}
\subsection{The Strichartz estimates}\label{sub:str}
We now prove the upper bounds in Theorems~\ref{thm:sharp,Strichartz,L4}
and~\ref{thm:sharp,short-time,L4}. The matching lower bounds and optimality of the time scales
are proved in Subsection~\ref{sub:sharp}.

\subsubsection{The nonsquare case}
The reduction below is the level-set argument from
\cite[Proof of Theorem 1.2]{HK24}. We recall it to indicate precisely
where the $B_{a,b}$-adapted counting estimate in Proposition
\ref{prop:count} enters.

\begin{proof}[Proof of the nonsquare case of Theorem \ref{thm:sharp,short-time,L4}]
By decomposing the real and imaginary parts of the Fourier coefficients
into their positive and negative parts and applying the triangle
inequality, it suffices to consider nonnegative Fourier coefficients.

Let $f=\widehat{P_{\le N}\phi}\ge0$, $m=\lfloor\log_{2}\#A_{N}\rfloor$, and $f_{0}=f$. Given $f_n$,
we enumerate all points of $A_{N}$, with an ordering which may depend on
$n$, so that
\begin{align*}
 f_n(\xi_{1}^{(n)})\ge f_n(\xi_{2}^{(n)})\ge\cdots
\end{align*}
and define
\begin{align*}
S_j^{(n)}=\{\xi_\ell^{(n)}:2^j\le\ell\le \min(2^{j+1}-1,\#A_{N})\},\quad
 \lambda_j^{(n)}=2^{j/2}f_n(\xi_{2^j}^{(n)}),\quad 0\le j\le m.
\end{align*}
Then from the above construction, we obtain
\begin{equation}\label{equ:lambda}
\n{\lambda^{(n)}}_{\ell^{2}}\lesssim\n{f_n}_{\ell^{2}}.
\end{equation}
% Indeed, $(\lambda_{0}^{(n)})^{2}=f_n(\xi_{1}^{(n)})^{2}$, while for
% $j\ge1$,
% \begin{align*}
% 2^j f_n(\xi_{2^j}^{(n)})^{2}\le 2\sum_{\ell=2^{j-1}}^{2^j-1}f_n(\xi_\ell^{(n)})^{2}.
% \end{align*}
% Summing in $j$ proves \eqref{equ:lambda}.

Let $\mathcal L_j^{(n)}$ be the family of lines containing at least
$2^{j/2+C}$ points of $S_j^{(n)}$, and let
$E_j^{(n)}\subset S_j^{(n)}$ be the points incident to two distinct
lines of $\mathcal L_j^{(n)}$. By Szemer\'edi--Trotter theorem in \eqref{equ:rich}, we get
\begin{align*}
\#\mathcal L_j^{(n)}
\lesssim
2^{j/2-3C}+2^{j/2-C}
\lesssim 2^{j/2-C}.
\end{align*}
Every point of $E_j^{(n)}$ determines a pair of distinct lines in
$\mathcal L_j^{(n)}$, and a pair of distinct lines has at most one
point of intersection. Therefore, we get
\begin{align*}
\#E_j^{(n)}\le \binom{\#\mathcal L_j^{(n)}}2
\lesssim 2^{j-2C}.
\end{align*}
Set
\begin{align*}
 f_{n+1}=f_n\bfone_{\cup_jE_j^{(n)}}.
\end{align*}
The monotone ordering gives
$f_n(\xi)\le2^{-j/2}\lambda_j^{(n)}$ on $S_j^{(n)}$. Hence, it holds that
\begin{align*}
\begin{aligned}
\n{f_{n+1}}_{\ell^{2}}^{2}
&=\sum_{j=0}^m\sum_{\xi\in E_j^{(n)}}f_n(\xi)^{2}\\
&\lesssim
\sum_{j=0}^m 2^{j-2C}
\left(2^{-j/2}\lambda_j^{(n)}\right)^{2}\\
&\lesssim
2^{-2C}\n{\lambda^{(n)}}_{\ell^{2}}^{2}
\lesssim
2^{-2C}\n{f_n}_{\ell^{2}}^{2}.
\end{aligned}
\end{align*}
After fixing $C$ sufficiently large, we have
\begin{align*}
\n{f_{n+1}}_{\ell^{2}}\le\frac12\n{f_n}_{\ell^{2}}.
\end{align*}
Then, we put $h_n=f_n-f_{n+1}$. The construction is telescoping and
$\n{f_n}_{\ell^{2}}\le2^{-n}\n f_{\ell^{2}}$, so
\begin{align*}
 f=\sum_{n=0}^\infty h_n \quad\text{in }\ell^{2}(\Z^{2}).
\end{align*}
Moreover, $h_n$ is pointwise dominated by
\begin{align*}
g_n=\sum_{j=0}^m\lambda_j^{(n)}2^{-j/2} \bfone_{S_j^{(n)}\setminus E_j^{(n)}}.
\end{align*}
The function $g_n$ satisfies Proposition
\ref{prop:count}. Indeed, if a point of
$S_j^{(n)}\setminus E_j^{(n)}$ lay on two distinct lines, each
containing at least $2^{j/2+C}$ points of the smaller set
$S_j^{(n)}\setminus E_j^{(n)}$, then both lines would lie in
$\mathcal L_j^{(n)}$, forcing that point to belong to
$E_j^{(n)}$, a contradiction.

Let us take $T=c/m$, where $0<c\le1$. Since $0\le h_n\le g_n$,
we have $h_n(Q)\le g_n(Q)$ for every parallelogram $Q$.
By Lemma~\ref{lemma:Oscillatory reduction} and \eqref{equ:dy0}, the zero-modulation
contribution is therefore
$O_{a,b}(Tm\n{\lambda^{(n)}}_{\ell^{2}}^{4})$.
For each dyadic $M\ge1$, \eqref{equ:dyM} gives
\begin{align*}
\begin{aligned}
\sum_{\tau\sim M}
\min\left\{T,\frac1{T\tau^{2}}\right\}
\sum_{Q\in\cQ_\tau^{a,b}}g_n(Q)
& \lesssimab
\min\left\{T,\frac1{TM^{2}}\right\}
M\n{\lambda^{(n)}}_{\ell^{2}}^{4}\\
& =
\min\{TM,(TM)^{-1}\}\n{\lambda^{(n)}}_{\ell^{2}}^{4}.
\end{aligned}
\end{align*}
Splitting the dyadic sum at $M\sim T^{-1}$, we have
\begin{align*}
        \sum_{M\in2^\N}\min\{TM,(TM)^{-1}\}\lesssim1.
\end{align*}
It follows that
\begin{align*}
\n{S(t)\mathcal F^{-1}h_n}_{L^{4}([0,T]\times\T^{2})}^{4}
\lesssimab
(Tm+1)\n{\lambda^{(n)}}_{\ell^{2}}^{4}.
\end{align*}
Since $Tm=c\le1$, we obtain
\begin{align*}
\n{S(t)\mathcal F^{-1}h_n}_{L^{4}([0,T]\times\T^{2})}
\lesssimab
\n{\lambda^{(n)}}_{\ell^{2}}
\lesssim
\n{f_n}_{\ell^{2}}.
\end{align*}
The series on the right is summable, so the partial sums of
$\sum_n S(t)\mathcal F^{-1}h_n$ converge in $L^{4}$.
By unitarity and the telescoping identity, they also converge to
$S(t)P_{\le N}\phi$ in $L^\infty([0,T];L^{2}(\T^{2}))$,
which identifies the $L^{4}$-limit. Minkowski's inequality then yields
\begin{align*}
\begin{aligned}
\n{S(t)P_{\le N}\phi}_{L^{4}([0,T]\times\T^{2})}
&\le
\sum_{n=0}^\infty
\n{S(t)\mathcal F^{-1}h_n}_{L^{4}([0,T]\times\T^{2})}\\
&\lesssimab
\sum_{n=0}^\infty\n{f_n}_{\ell^{2}}
\lesssim
\n{\phi}_{L^{2}}.
\end{aligned}
\end{align*}
Since $m\sim\log N$, decreasing the constant $c=c(a,b)$ in the
definition of $T_{N}$ if necessary proves \eqref{equ:sharp,short-time,L4}.
\end{proof}

\begin{remark}[Dependence on the size of the Fourier support]
\label{rem:cardinality}
Assume that $-ab$ is not a square. The proof above uses the frequency
set only through its cardinality. Thus, for every nonempty finite
$E\subset\Z^{2}$, the same argument,
with $m=\max\{1,\lfloor\log_{2}\#E\rfloor\}$, gives
\begin{align*}
 \n{S(t)P_E\phi}_{L^{4}(I\times\T^{2})}\lesssimab\n\phi_{L^{2}},
 \quad |I|\le\frac{c_{a,b}}{\log(2+\#E)}.
\end{align*}
For $\#E=1$, this follows directly from the single Fourier mode.
The estimate holds on arbitrary time translates of the interval,
since $S(t_{0})$ preserves the Fourier support and the $L^{2}$-norm.
Partitioning $[0,2\pi]$ yields the corresponding loss
$\bigl(\log(2+\#E)\bigr)^{1/4}$, independently of the location or
diameter of $E$.
\end{remark}

\subsubsection{The semiclassical estimate in the square case}
In the square case, when $-ab$ is a square, the fixed-time upper
bound in Theorem~\ref{thm:sharp,Strichartz,L4} follows from
Wang's estimate \cite[Theorem 2.1]{Wang13} and the rational-torus
covering argument of Guo--Oh--Wang \cite[Subsection 1.2]{GOW14}.
Here we prove the corresponding short-time bound in
Theorem~\ref{thm:sharp,short-time,L4} on intervals of length
$cN^{-1}$. We begin with a semiclassical kernel estimate
valid for every nondegenerate form \eqref{equ:disp}.

\begin{lemma}[Semiclassical kernel bound]
\label{lemma:semi}
Let $a,b\in\Z\setminus\{0\}$ and $\mu\in C_c^\infty(\R^{2})$.
For $N\ge1$, we define
\begin{align*}
K_{N}^\mu(t,x)
=\sum_{k\in\Z^{2}}\mu(k/N)e^{i(k\cdot x-t\omega(k))}.
\end{align*}
There exists $c=c(a,b,\mu)>0$ such that, uniformly for $N\ge1$
and $x\in\T^{2}$,
\begin{equation}\label{equ:semi-disp}
\abs{K_{N}^\mu(t,x)}
\lesssim_{a,b,\mu}
\min\{N^{2},\abs{t}^{-1}\},
\quad 0<\abs{t}\le cN^{-1}.
\end{equation}
\end{lemma}

\begin{proof}
We may take $c\le1$.
The trivial estimate gives $\abs{K_{N}^\mu(t,x)}\lesssim_\mu N^{2}$,
which proves \eqref{equ:semi-disp} when
$0<\abs{t}\le N^{-2}$. We therefore assume
\begin{align*}
N^{-2}\le\abs{t}\le cN^{-1}.
\end{align*}
By Poisson summation, up to harmless normalizing constants, we get
\begin{align*}
K_{N}^\mu(t,x)
=N^{2}\sum_{n\in\Z^{2}}I_n(t,x),
\end{align*}
where
\begin{align*}
I_n(t,x)
=\int_{\R^{2}}\mu(\eta)e^{i\Phi_n(\eta)}\,d\eta,
\quad
\Phi_n(\eta)
=N(x+2\pi n)\cdot\eta
-tN^{2}(a\eta_{1}^{2}+b\eta_{2}^{2}).
\end{align*}
Choose $C_{0}$ sufficiently large depending on $a,b$ and the
support of $\mu$. There are only $O(1)$ indices satisfying
\begin{align*}
\abs{x+2\pi n}\le C_{0}\abs{t}N.
\end{align*}
Set $D=\operatorname{diag}(a,b)$. For each such index, the critical point
\begin{align*}
\eta_n=\frac{D^{-1}(x+2\pi n)}{2tN}
\end{align*}
lies in a fixed bounded set. Completing the square gives
\begin{align*}
 \Phi_n(\eta)=\Phi_n(\eta_n)-tN^{2}(\eta-\eta_n)^TD(\eta-\eta_n).
\end{align*}
The translated amplitudes $\mu(\,\cdot+\eta_n)$ have uniformly bounded
supports and derivatives. Since $\det D=ab\ne0$ and
$\abs{t}N^{2}\ge1$, two-dimensional quadratic stationary phase applies
regardless of the signature of $D$ and yields
\begin{align*}
        \abs{I_n(t,x)}
        \lesssim_{a,b,\mu}(\abs{t}N^{2})^{-1}.
\end{align*}

For the remaining indices, since
$\abs{x+2\pi n}>C_{0}\abs{t}N$, we have
\begin{align*}
        \abs{\nabla_\eta\Phi_n(\eta)}
        \gtrsim N\abs{x+2\pi n}
        \quad(\eta\in\operatorname{supp}\mu).
\end{align*}
Repeated application of the integration-by-parts operator
\begin{align*}
L_n=\frac{\nabla_\eta\Phi_n\cdot\nabla_\eta} {i\abs{\nabla_\eta\Phi_n}^{2}},
        \quad L_ne^{i\Phi_n}=e^{i\Phi_n},
\end{align*}
gives, for every $L\ge1$,
\begin{align*}
 \abs{I_n(t,x)}\lesssim_{L,a,b,\mu}\bigl(N\abs{x+2\pi n}\bigr)^{-L},
\end{align*}
where we used
$\abs{x+2\pi n}\gtrsim\abs{t}N$ to control derivatives of the
coefficients of $L_n$. Fix $L>2$. Among the remaining indices,
those with $\abs{x+2\pi n}\le1$ are $O(1)$ in number and satisfy
$N\abs{x+2\pi n}\gtrsim\abs{t}N^{2}$. For those with
$\abs{x+2\pi n}>1$, the sum of $\abs{x+2\pi n}^{-L}$ is
uniformly bounded in $x$. Hence, we arrive at
\begin{align*}
N^{2}\sum_{\substack{n\in\Z^{2}\\
|x+2\pi n|>C_{0}|t|N}}\abs{I_n(t,x)}&\lesssim_{L,a,b,\mu}
        N^{2}(\abs{t}N^{2})^{-L}+N^{2-L}
        \lesssim\abs{t}^{-1},
\end{align*}
where we used $\abs{t}N^{2}\ge1$ and $\abs{t}\le1$.
Together with the contribution of the $O(1)$ near indices, this proves
\eqref{equ:semi-disp}.
\end{proof}

At frequencies of size $N$, the group velocity is $O(N)$, so wave
packets travel only a bounded distance over times of order $N^{-1}$.
Correspondingly, only a uniformly bounded number of translated Euclidean
kernels in the Poisson sum can have stationary points on the cutoff
support, while the remaining terms are summable by nonstationary phase.
Thus \eqref{equ:semi-disp} retains the Euclidean dispersive bound on this
time scale.

\begin{proof}[Proof of the square case of Theorem \ref{thm:sharp,short-time,L4}]
The argument below applies in fact to every nondegenerate form
\eqref{equ:disp}. Choose a real-valued
$\chi\in C_c^\infty(\R^{2})$ which equals one on $[-1,1]^{2}$, and
let $\widetilde P_{N}$ be the multiplier with symbol $\chi(k/N)$.
Set
\begin{align*}
U_{N}(t)=S(t)\widetilde P_{N}.
\end{align*}
The operator $U_{N}(t)U_{N}(s)^*$ has convolution kernel
\begin{align*}
\sum_{k\in\Z^{2}}\abs{\chi(k/N)}^{2}
e^{i(k\cdot x-(t-s)\omega(k))}.
\end{align*}
Applying Lemma \ref{lemma:semi} with
$\mu=\abs{\chi}^{2}$, we obtain, whenever
$0<\abs{t-s}\le cN^{-1}$,
\begin{align*}
\n{U_{N}(t)U_{N}(s)^*}_{L^{1}_x\to L^\infty_x}
\lesssimab \abs{t-s}^{-1}.
\end{align*}
On the other hand, $L^{2}$-unitarity of $S(t)$ and boundedness of
the cutoff give
\begin{align*}
\n{U_{N}(t)U_{N}(s)^*}_{L^{2}_x\to L^{2}_x}\lesssim1.
\end{align*}
Interpolation yields
\begin{equation}
\label{equ:TTx}
\n{U_{N}(t)U_{N}(s)^*}_{L^{4/3}_x\to L^{4}_x}
\lesssimab \abs{t-s}^{-1/2}.
\end{equation}

Let $I$ be any interval with $\abs I\le cN^{-1}$. By
\eqref{equ:TTx} and the one-dimensional
Hardy--Littlewood--Sobolev inequality,
\begin{align*}
\bbn{\int_IU_{N}(t)U_{N}(s)^*F(s)\,ds}_{L^{4}_{t,x}(I\times\T^{2})}
\lesssimab&
\bbn{\int_I\abs{t-s}^{-1/2}
\n{F(s)}_{L^{4/3}_x}\,ds}_{L^{4}_t(I)}\\
\lesssimab&
\n F_{L^{4/3}_{t,x}(I\times\T^{2})}.
\end{align*}
The $TT^*$ argument therefore gives
\begin{align*}
\n{S(t)\widetilde P_{N}\phi}_{L^{4}(I\times\T^{2})}\lesssimab \n{\phi}_{L^{2}}.
\end{align*}
Applying this to $\phi=P_{\le N}\phi$, for which
$\widetilde P_{N}\phi=\phi$, proves the square-case estimate in
\eqref{equ:sharp,short-time,L4}. Notice that the argument applies to every
time translate of $I$.
\end{proof}

\subsection{Sharpness of the estimates}\label{sub:sharp}
We establish matching lower bounds using box data and data supported
on a null line, then deduce the optimality of the short-time scales.
\begin{lemma}[The box lower bound]\label{lemma:box}
Assume $ab\ne0$. For every dyadic $N\ge2$, let
\begin{align*}
        \wh{\phi_{N}}(k)=\bfone_{A_{N}}(k).
\end{align*}
Then
\begin{align*}
\n{S(t)\phi_{N}}_{L^{4}([0,2\pi]\times\T^{2})} \gtrsimab (\log N)^{1/4}\n{\phi_{N}}_{L^{2}(\T^{2})}.
\end{align*}
\end{lemma}
\begin{proof}
Let
\begin{align*}
        u_{N}(t,x)=S(t)\phi_{N}(x)
        =
        \sum_{k\in A_{N}}e^{i(x\cdot k-t\omega(k))}.
\end{align*}
By spatial orthogonality, \eqref{equ:res}, and the integrality of the
modulation, integration over $[0,2\pi]$ gives
\begin{align*}
\n{u_{N}}_{L^{4}([0,2\pi]\times\T^{2})}^{4}=(2\pi)^3\#\cR_{N},
\end{align*}
where $\cR_{N}=\cQ_{0}^{a,b}\cap A_{N}^{4}$ is the set of resonant
parallelograms with vertices in $A_{N}$.

It is enough to count a subfamily. For
$\xi\in A_{N/2}$ and $(u,v)\in\mathcal P_{N/4}^{a,b}$, we set
\begin{align*}
(\xi_{1},\xi_{2},\xi_{3},\xi_{4})=(\xi,\xi-u,\xi-u-v,\xi-v).
\end{align*}
All four vertices belong to $A_{N}$, the parallelogram relation holds,
and the resonance condition is $B_{a,b}(u,v)=0$. Moreover,
$(\xi,u,v)$ is recovered from the ordered quadruple, so this
construction is injective. Hence, we have
\begin{align*}
\#\cR_{N}\gtrsim N^{2}\,\#\mathcal P_{N/4}^{a,b},
\end{align*}
where
\begin{align*}
\mathcal P_{M}^{a,b}=\{(u,v)\in A_{M}^{2}:\ B_{a,b}(u,v)=0\}.
\end{align*}
We claim that
\begin{align*}
\#\mathcal P_{M}^{a,b}\gtrsimab M^{2}\log M.
\end{align*}
For simplicity, we write $d=\gcd(|a|,|b|)$, $\alpha=|a|/d$, and $\beta=|b|/d$, so $(\alpha,\beta)=1$. For positive integers $g,\ell,p,q$, with $(p,q)=1$, we define
\begin{align*}
 u=(\beta gp,\alpha gq),\quad
 v=(q\ell,-\operatorname{sgn}(ab)p\ell).
\end{align*}
Then $B_{a,b}(u,v)=0$. This parametrization is injective. From $u$,
one recovers
\begin{align*}
g=\gcd(u_{1}/\beta,u_{2}/\alpha),\quad p=u_{1}/(\beta g),\quad q=u_{2}/(\alpha g),
\end{align*}
and then $\ell$ is recovered from $v$. The conditions
$u,v\in A_{M}$ are guaranteed by
\begin{align*}
g\le \frac{cM}{\max(p,q)},\quad
\ell\le \frac{cM}{\max(p,q)}
\end{align*}
with $c=c(a,b)>0$ sufficiently small. Put
$r=\max(p,q)$. If $r\le cM$, then the number of admissible pairs
$(g,\ell)$ is
\begin{align*}
\left\lfloor\frac{cM}{r}\right\rfloor^{2}
\gtrsimab \frac{M^{2}}{r^{2}},
\end{align*}
where we used $\lfloor x\rfloor\ge x/2$ for $x\ge1$. Therefore, we get
\begin{align*}
\#\mathcal P_{M}^{a,b}\gtrsimab M^{2} \sum_{\substack{1\le p,q\le cM\,(p,q)=1}}\frac1{\max(p,q)^{2}}
\gtrsim   M^{2}\sum_{m\le cM}\frac{\varphi(m)}{m^{2}}
 \gtrsim  M^{2}\log M.
\end{align*}
For completeness, we justify the last two estimates. For $m\ge2$,
the number of ordered positive coprime pairs $(p,q)$ with
$\max(p,q)=m$ is $2\varphi(m)$. Thus, we derive
\begin{align*}
\sum_{\substack{1\le p,q\le cM\ (p,q)=1}}
\frac1{\max(p,q)^{2}}
\gtrsim \sum_{m\le cM}\frac{\varphi(m)}{m^{2}}.
\end{align*}
Let $\mu_{\rm M}$ denote the M\"obius function. The identity
\begin{align*}
\frac{\varphi(m)}m=\sum_{d\mid m}\frac{\mu_{\rm M}(d)}d
\end{align*}
and the substitution $m=dr$ give
\begin{align*}
\sum_{m\le X}\frac{\varphi(m)}{m^{2}}
&=\sum_{d\le X}\frac{\mu_{\rm M}(d)}{d^{2}}
\sum_{r\le X/d}\frac1r
=\frac1{\zeta(2)}\log X+O(1).
\end{align*}
Here, we used
$\sum_{d\ge1}\mu_{\rm M}(d)d^{-2}=1/\zeta(2)$ and the absolute
convergence of $\sum_{d\ge1}(1+\log d)d^{-2}$. This proves the
claimed logarithmic lower bound.

Thus, $\#\cR_{N}\gtrsimab N^{4}\log N$. Since
\begin{align*}
 \n{\phi_{N}}_{L^{2}(\T^{2})}^{2}\sim\#A_{N}\sim N^{2},
\end{align*}
the desired lower bound follows.
The argument above was written for sufficiently large $N$. The
finitely many remaining dyadic values are absorbed by decreasing the
implicit constant.
\end{proof}

\begin{lemma}[Null-line lower bound]
\label{lemma:null-lb}
Assume that $-ab$ is a square. For every dyadic $N\ge2$ and every
$0<T\le2\pi$, there exists
$\psi_{N}\ne0$ with Fourier support in $A_{N}$ such that
\begin{align*}
\n{S(t)\psi_{N}}_{L^{4}([0,T]\times\T^{2})}\gtrsimab (TN)^{1/4}\n{\psi_{N}}_{L^{2}(\T^{2})}.
\end{align*}
Consequently, both the $N^{1/4}$ loss in
\eqref{equ:sharp,Strichartz,L4} and the $N^{-1}$ time scale in
\eqref{equ:sharp,short-time,L4} are sharp.
\end{lemma}

\begin{proof}
Since $-ab$ is a square and $ab<0$, the equation
\begin{align*}
 a k_{1}^{2}+b k_{2}^{2}=0
\end{align*}
has a nonzero primitive integer solution $v$. Let
$N\ge4\n{v}_{\ell^\infty}$ for the moment, and set
\begin{align*}
K_{N}=\left\lfloor\frac{N}{2\n{v}_{\ell^\infty}}\right\rfloor,\quad
L_{N}=\{nv:\ |n|\le K_{N}\}\subset A_{N},
\end{align*}
and set $\wh{\psi_{N}}=\bfone_{L_{N}}$. Notice that
$K_{N}\sim_{a,b}N$. Since $\omega(k)=0$ for all $k\in L_{N}$,
\begin{align*}
 S(t)\psi_{N}(x)=\sum_{|n|\le K_{N}}e^{in v\cdot x}
\end{align*}
is independent of $t$. Because $v$ is primitive, the torus
homomorphism
\begin{align*}
 x\longmapsto v\cdot x\pmod{2\pi}
\end{align*}
pushes Haar measure on $\T^{2}$ forward to Haar measure on $\T$.
Consequently, up to harmless powers of $2\pi$, the spatial norms above
are those of the one-dimensional Dirichlet kernel
\begin{align*}
 D_{K_{N}}(y)=\sum_{|n|\le K_{N}}e^{iny}.
\end{align*}
Orthogonality gives
\begin{align*}
 \n{D_{K_{N}}}_{L^{2}(\T)}\sim K_{N}^{1/2}.
\end{align*}
Moreover,
\begin{align*}
\n{D_{K_{N}}}_{L^{4}(\T)}^{4}
\sim
\#\{(n_{1},n_{2},n_{3},n_{4})\in([-K_{N},K_{N}]\cap\Z)^{4}:
n_{1}+n_{3}=n_{2}+n_{4}\}
\sim K_{N}^3.
\end{align*}
For the upper bound, any three indices determine the fourth; for the
lower bound, restrict the first three indices to
$[-K_{N}/3,K_{N}/3]$, so the determined fourth remains in
$[-K_{N},K_{N}]$. We have therefore proved
\begin{align*}
\n{\psi_{N}}_{L^{2}(\T^{2})}\simab N^{1/2}, \quad
 \n{S(t)\psi_{N}}_{L^{4}([0,T]\times\T^{2})}\simab T^{1/4}N^{3/4}.
\end{align*}
The ratio is comparable to $(TN)^{1/4}$, with constants depending
on $a,b$, as claimed.
For the finitely many $2\le N<4\n{v}_{\ell^\infty}$, take
$\psi_{N}\equiv1$. Its ratio is a constant multiple of $T^{1/4}$,
which dominates $c_{a,b}(TN)^{1/4}$ after decreasing
$c_{a,b}>0$. This completes the proof for every dyadic $N\ge2$.
\end{proof}

\begin{proof}[Completion of Theorems \ref{thm:sharp,Strichartz,L4} and \ref{thm:sharp,short-time,L4}] 
The short-time upper bounds were proved in Subsection~\ref{sub:str},
and the fixed-time upper bounds follow by time partitioning.

The sharpness of these fixed-time bounds follows from
Lemma~\ref{lemma:box} in the nonsquare case and
Lemma~\ref{lemma:null-lb} in the square case, which give the matching
$(\log N)^{1/4}$ and $N^{1/4}$ lower bounds, respectively.
The short-time scales are also optimal by contradiction. We omit details.
\end{proof}

\section{Global well-posedness for the cubic NLS}\label{section:Global well-posedness}
In this section, we prove Theorem \ref{thm:gwp}.  Herr--Kwak
\cite[Section 4]{HK24} showed that a lossless cube-localized $L^{4}$
estimate on the logarithmic time scale yields small-mass global
well-posedness for the cubic equation.  That nonlinear implication is
robust under a change of the real quadratic dispersion.  
Throughout this section, implicit constants may depend on the fixed
regularity $s>0$ and coefficients $a,b$, but are independent of the
iteration scale $N$, unless otherwise stated.

Consider
\begin{equation}\label{equ:nls}
\begin{cases}
 i\pa_{t} u-\Omega u=\mu |u|^{2}u,\\
 u(0)=u_{0},
\end{cases}
\quad \mu\in\{\pm1\}.
\end{equation}
The mass is conserved for smooth solutions
$\n{u(t)}_{L^{2}(\T^{2})}=\n{u_{0}}_{L^{2}(\T^{2})}$.
Solutions below are understood in the Duhamel sense. The sign $\mu$
plays no role in the estimates.

\medskip

For dyadic $N$, let $\mathcal C_{N}$ denote the partition of $\Z^{2}$ into cubes of side length $N$.
\begin{lemma}[Galilean covariance]
\label{lemma:freq}
For $k_{0}\in\Z^{2}$, let $M_{k_{0}}g(x)=e^{ik_{0}\cdot x}g(x)$ and
\begin{align*}
        v(k_{0})=2(ak_{0,1},bk_{0,2}).
\end{align*}
Then
\begin{equation}\label{equ:gal}
 S(t)M_{k_{0}}g(x) =e^{ik_{0}\cdot x-it\omega(k_{0})}S(t)g(x-tv(k_{0})).
\end{equation}
Consequently, every estimate in Theorem \ref{thm:sharp,short-time,L4} remains
valid, with the same constant, for Fourier support in any translate
of $A_{N}$.
\end{lemma}

\begin{proof}
The identity follows from
\begin{align*}
\omega(k_{0}+n)=\omega(k_{0})+\omega(n)+2B_{a,b}(k_{0},n).
\end{align*}
For each fixed $t$, modulation, multiplication by a scalar phase,
and translation on $\T^{2}$ preserve every spatial $L^{p}$-norm.
\end{proof}

We record two consequences of
Theorem~\ref{thm:sharp,short-time,L4} and Lemma~\ref{lemma:freq}.
The first follows by the transference argument in
\cite[Proof of Lemma 4.3]{HK24}, with $S(t)$ in place of $e^{it\Delta}$.
The second follows from the first by time partitioning and dyadic
summation. We omit the proofs.

\begin{lemma}[Strichartz estimates in \texorpdfstring{$Y^s$}{Ys}]
\label{lemma:Y4}
Assume that $-ab$ is not a square. Then the following estimates hold.
\begin{enumerate}
\item For every dyadic $N\ge2$, every interval $I\subset\R$ with
$|I|\le c(\log N)^{-1}$, every $C\in\mathcal C_N$, and every
$u\in Y^0(I)$,
\begin{align*}
 \n{P_Cu}_{L^{4}(I\times\T^{2})}
 \lesssimab \n{P_Cu}_{Y^0(I)}.
\end{align*}
\item For every $s>0$, every bounded interval $I\subset\R$, and every
$u\in Y^s(I)$,
\begin{align*}
 \n{u}_{L^{4}(I\times\T^{2})}
 \lesssim_{a,b,s,I}\n{u}_{Y^s(I)}.
\end{align*}
\end{enumerate}
\end{lemma}

If $M\ge N$ and $C\in\mathcal C_{M}$, partitioning an interval $I_{N}$ of length $(\log N)^{-1}$ into $O(\log M/\log N)$ intervals of length $(\log M)^{-1}$ gives
\begin{equation}\label{equ:Y4-long}
\n{\mathbf 1_{I_{N}}P_Cu}_{L^{4}_{t,x}}\lesssimab
\left(1+\frac{\log M}{\log N}\right)^{1/4}\n{P_Cu}_{Y^0(I_{N})}.
\end{equation}
Here and below constants absorb the fixed constant $c=c(a,b)$ in
the admissible time scale, and restriction norms are monotone when the
time interval is shortened.

\medskip

To state the trilinear estimate, fix $s>0$. For dyadic $N\ge2$, we set
\begin{align*}
        I_{N}=[0,c(\log N)^{-1}]
\end{align*}
and define
\begin{align*}
\n{u}_{Z_{N}} = \n{u}_{Y^0(I_{N})}+N^{-s}\n{u}_{Y^{s}(I_{N})}.
\end{align*}
The following estimate is obtained by the argument of
\cite[Proof of Lemma 4.4]{HK24}, with Lemma~\ref{lemma:Y4}(1)
and \eqref{equ:Y4-long} providing the Strichartz inputs.
The Duhamel duality and frequency summation argument applies to the
present dispersion. Using ordinary dyadic scales gives the result
for every fixed $s>0$, with constants allowed to depend on $s,a,b$.

\begin{lemma}[Trilinear estimate]
\label{lemma:tri}
Assume that $-ab$ is not a square. For $N$ sufficiently large
depending on $s,a,b$,
\begin{align*}
        \n{\mathcal I_{0}(u_{1}u_{2}u_{3})}_{Z_{N}}
        \lesssim_{s,a,b}
        \n{u_{1}}_{Z_{N}}\n{u_{2}}_{Z_{N}}\n{u_{3}}_{Z_{N}},
\end{align*}
where any of the factors $u_j$ may be replaced by $\overline{u_j}$.
\end{lemma}

\medskip

Now, we are ready to discuss the proof of
Theorem~\ref{thm:gwp}. 

\begin{proof}[Proof of Theorem \ref{thm:gwp}]
With Lemma~\ref{lemma:tri} at hand, Theorem~\ref{thm:gwp}
follows by the contraction and logarithmic iteration argument
of \cite[Proof of Theorem 1.4]{HK24}. We omit the details.
\end{proof}

\begin{ackno}\normalfont 
S.S. was supported in part by the NSF of China under Grant 12501322, and Anhui Provincial NSF 2508085QA001.
Y.W. was supported by the EPSRC Mathematical Sciences Small Grant (grant no. UKRI1116).
\end{ackno}

\noindent\textbf{Generative AI Statement.}
During the preparation of this work, the authors used ChatGPT (OpenAI) and DeepSeek to assist with language polishing and LaTex formatting. All AI-assisted content was reviewed, verified, and revised as necessary by the authors. The authors take full responsibility for the accuracy, originality, and entire content of the manuscript.

\noindent\textbf{Data Availability Statement}
Data sharing is not applicable to this article as no datasets were generated or analysed during the current study.

\noindent\textbf{Conflict of Interest}
The authors declare that they have no conflict of interest.

\appendix

\section{Typed-rectangle bookkeeping for the dyadic count}\label{appendix:Typed-rectangle}
This appendix proves Lemmas~\ref{lemma:type,red} and~\ref{lemma:type,est},
used in the proof of Proposition~\ref{prop:count}. These extend
\cite[Lemmas 3.2 and 3.3]{HK24}, which treat $(a,b)=(1,1)$, to
general nonzero integral coefficients $(a,b)$ with $-ab$ nonsquare.
Their proofs follow the same scheme, with
Lemmas~\ref{lemma:corner}--\ref{lemma:divisor side} supplying the
counting estimates in the present geometric setting.
For completeness, we provide the detailed arguments.
Readers satisfied that these estimates preserve the counting
structure underlying the analytic argument may skip this appendix.

\subsection{Reduction to typed resonant rectangles}\label{appendix:Reduction}
Let $S_{0},\ldots,S_m$ be the sets in Proposition~\ref{prop:count}.
For $\alpha,\beta\in\{1,2,3\}$, denote by $\cQ_{0,\alpha,\beta}^{a,b}$ the set of rectangles $Q=(\xi_{1},\xi_{2},\xi_{3},\xi_{4})\in\cQ_{0}^{a,b}$, with four distinct vertices in $\cup_jS_j$, such that $\xi_{1},\xi_{2}$ are type $\alpha$ vertices and $\xi_{3},\xi_{4}$ are type $\beta$ vertices.
\begin{lemma}[Reduction to typed resonant rectangles]
\label{lemma:type,red}
Let $f$ be as in Proposition~\ref{prop:count}, and assume that
$-ab$ is not a square.
For every integer $\tau\ge0$,
\begin{equation}\label{equ:reduction to typed}
\sum_{Q\in\cQ_\tau^{a,b}}f(Q)\lesssim\max_{\alpha,\beta=1,2,3}
\sum_{\substack{Q\in\cQ_{0,\alpha,\beta}^{a,b}\\
d_{a,b}(\xi_{1}-\xi_{4})\mid \tau}}f(Q) + \n f_{\ell^{2}(\Z^{2})}^{4}.
\end{equation}
\end{lemma}

\begin{proof}[Proof of Lemma \ref{lemma:type,red}]
We first dispose of all parallelograms having a repeated vertex. If two
adjacent vertices coincide, the parallelogram relation forces the other
two adjacent vertices to coincide as well. The total contribution of
all such configurations is therefore, up to an absolute multiplicity,
bounded by
\begin{align*}
\sum_{\eta,\zeta}f(\eta)^{2}f(\zeta)^{2}=\n f_{\ell^{2}}^{4}.
\end{align*}
It remains to consider a repeated pair of opposite vertices. If
$\xi_{1}=\xi_{3}=x$, write
$\xi_{2}=x+h$ and $\xi_{4}=x-h$. By Cauchy--Schwarz,
\begin{align*}
\sum_{x,h}f(x)^{2}f(x+h)f(x-h)
&\le
\sum_x f(x)^{2}
\bigg(\sum_hf(x+h)^{2}\bigg)^{1/2}
\bigg(\sum_hf(x-h)^{2}\bigg)^{1/2}=\n f_{\ell^{2}}^{4}.
\end{align*}
The case $\xi_{2}=\xi_{4}$ is identical. Thus every repeated-vertex
configuration has the required bound.
We may therefore assume that the four vertices are distinct and put
$\xi=\xi_{1}-\xi_{4}\ne0$. For such a parallelogram,
\begin{align*}
        \tau=2\abs{B_{a,b}(\xi_{1}-\xi_{2},\xi)}.
\end{align*}
Lemma \ref{lemma:Borth} shows that the exact divisibility
condition is
\begin{align*}
2d_{a,b}(\xi)\mid\tau.
\end{align*}
For the upper bound below we use the weaker necessary condition
$d_{a,b}(\xi)\mid\tau$. In particular, every nonempty modulation
class has even $\tau$; when $\tau$ is odd the assertion is trivial.

For $\sigma\in\Z$, we set
\begin{align*}
        E_\xi^{\sigma}
        =
        \{(\eta,\eta-\xi): B_{a,b}(\eta,\xi)=\sigma\}.
\end{align*}
The two opposite sides $(\xi_{1},\xi_{4})$ and $(\xi_{2},\xi_{3})$ lie in
levels $E_\xi^{\sigma_{1}}$ and $E_\xi^{\sigma_{2}}$ with
$\sigma_{1}-\sigma_{2}=\pm\tau/2$.
Writing
\begin{align*}
 A_\xi(\sigma)=\sum_{(\eta,\eta-\xi)\in E_\xi^{\sigma}}f(\eta)f(\eta-\xi),
\end{align*}
Cauchy's inequality gives, for fixed $\xi$,
\begin{align*}
\begin{aligned}
\sum_{\substack{Q\in\cQ_\tau^{a,b}\\ \xi_{1}-\xi_{4}=\xi}} f(Q)\lesssim
 \sum_{\sigma_{1}-\sigma_{2}=\pm\tau/2}
 A_\xi(\sigma_{1})A_\xi(\sigma_{2}) \lesssim
 \sum_\sigma A_\xi(\sigma)^{2}.
\end{aligned}
\end{align*}
Split $E_\xi^{\sigma}$ into the nine subclasses
$E_{\xi,\alpha,\beta}^{\sigma}$, according to the types of the
$(a,b)$-crosses at $\eta$ and $\eta-\xi$ with directions
$\xi$ and $\xi^{\perp_{a,b}}$. Since there are only nine classes,
\begin{align*}
 \sum_\sigma A_\xi(\sigma)^{2}\lesssim
 \sum_{\alpha,\beta=1}^3\sum_\sigma
 \Bigg(
 \sum_{(\eta,\eta-\xi)\in E_{\xi,\alpha,\beta}^{\sigma}}
 f(\eta)f(\eta-\xi)
 \Bigg)^{2}.
\end{align*}
If two distinct segments $(\eta,\eta-\xi)$ and
$(\eta',\eta'-\xi)$ lie in the same $E_{\xi,\alpha,\beta}^{\sigma}$,
then
\begin{align*}
        B_{a,b}(\eta-\eta',\xi)=0.
\end{align*}
The four endpoints are distinct. Indeed, equality of the two segments
has already been excluded. Any remaining collision would give
$\eta-\eta'=\pm\xi$; combining this with the preceding orthogonality
would imply $B_{a,b}(\xi,\xi)=0$, contrary to
Lemma \ref{lemma:Borth}. Hence
\begin{align*}
(\eta,\eta',\eta'-\xi,\eta-\xi)\in\cQ_{0,\alpha,\beta}^{a,b}.
\end{align*}
Therefore, for each $\alpha,\beta$,
\begin{align*}
\sum_{\substack{\xi\in\Z^{2}\setminus\{0\}\\ d_{a,b}(\xi)\mid\tau}}
\sum_\sigma
\Bigg(
\sum_{(\eta,\eta-\xi)\in E_{\xi,\alpha,\beta}^{\sigma}}
f(\eta)f(\eta-\xi)
\Bigg)^{2}
&\quad\lesssim
\sum_{\substack{Q\in\cQ_{0,\alpha,\beta}^{a,b}\\
d_{a,b}(\xi_{1}-\xi_{4})\mid\tau}}f(Q)
+\sum_{\eta,\zeta\in\Z^{2}} f(\eta)^{2}f(\zeta)^{2}.
\end{align*}
Taking the maximum over $\alpha,\beta$ and using
$\sum_{\eta,\zeta}f(\eta)^{2}f(\zeta)^{2}=\n f_{\ell^{2}}^{4}$ proves
\eqref{equ:reduction to typed}.
\end{proof}

\subsection{The type-by-type estimates}\label{appendix:type-by-type}
\begin{lemma}[Typed rectangle estimates]
\label{lemma:type,est}
Let $f$ be as in Proposition \ref{prop:count}, and
assume that $-ab$ is not a square.
If $(\alpha,\beta)\ne(2,2)$, then
\begin{equation}\label{equ:type,off}
 \sum_{Q\in\cQ_{0,\alpha,\beta}^{a,b}}f(Q)\lesssimab
 \n{\lambda}_{\ell^{2}}^{4}.
\end{equation}
For $(\alpha,\beta)=(2,2)$,
\begin{equation}\label{equ:type22}
\sum_{Q\in\cQ_{0,2,2}^{a,b}}f(Q)\lesssimab
 m\n{\lambda}_{\ell^{2}}^{4}
\end{equation}
and
\begin{equation}\label{equ:type22-div}
\sum_{Q\in\cQ_{0,2,2}^{a,b}}\frac{1}{d_{a,b}(\xi_{1}-\xi_{4})}f(Q)\lesssimab
\n{\lambda}_{\ell^{2}}^{4}.
\end{equation}
\end{lemma}

The extra divisor weight in \eqref{equ:type22-div} supplies the
summability needed for nonzero modulation: after averaging in $\tau$,
it removes the common medium-line scale and hence the logarithmic
loss present in \eqref{equ:type22}.

\begin{proof}[Proof of Lemma \ref{lemma:type,est}]
We give the four cases of the argument in
\cite[Proof of Lemma 3.3]{HK24}, using the counting bounds established
above. All implicit constants may depend on the fixed richness threshold
$C$. We extend $\lambda_j$ by zero outside $0\le j\le m$.
We repeatedly use the elementary convolution bound
\begin{equation}\label{equ:app-young}
 \sum_{j,k\in\Z}2^{-\delta|j-k|}\lambda_j\lambda_k
 \lesssim_\delta\n\lambda_{\ell^{2}}^{2},
 \quad \delta>0,
\end{equation}
which follows from Young's inequality and
$(2^{-\delta|n|})_{n\in\Z}\in\ell^{1}(\Z)$.

\medskip\noindent
\emph{Case 1: at least one of $\alpha,\beta$ equals $1$.}
By a cyclic relabeling it suffices to consider $\alpha=1$.
For each rich vertex $x$, the hypothesis of Proposition
\ref{prop:count} specifies a unique rich line through $x$.
Let $v\in\Z^{2}\setminus\{0\}$ be a direction vector of that line.
Given the opposite vertex $z$, a neighboring vertex on this line has
form $y=x+tv$. Orthogonality of the two sides requires
\begin{align*}
 B_{a,b}(v,z-x-tv)=0,
 \quad
 t=\frac{B_{a,b}(v,z-x)}{B_{a,b}(v,v)}.
\end{align*}
The denominator is nonzero by Lemma~\ref{lemma:Borth}. Thus $y$, and
then the fourth vertex $x+z-y$, are uniquely determined, if they exist.
Allowing the two possible assignments of the neighbors gives bounded
multiplicity. Since both $\xi_{1}$ and $\xi_{2}$ are rich in this case,
this applies to either opposite pair. Cauchy--Schwarz therefore gives
\begin{align*}
 \sum_{Q\in\cQ_{0,1,\beta}^{a,b}}f(Q)
 \le
 \bigg(\sum_Q f(\xi_{1})^{2}f(\xi_{3})^{2}\bigg)^{1/2}
 \bigg(\sum_Q f(\xi_{2})^{2}f(\xi_{4})^{2}\bigg)^{1/2}
 \lesssim\n f_{\ell^{2}}^{4}
 \le\n\lambda_{\ell^{2}}^{4}.
\end{align*}
This is the precise reconstruction step for which the absence of
rational null directions is essential; no lower bound on the Euclidean
angle between the two side directions is needed.

\medskip\noindent
\emph{Case 2: $(\alpha,\beta)=(2,2)$.}
Fix $\vec j=(j_{1},j_{2},j_{3},j_{4})$ and
$\vec a=(a_{1},a_{2},a_{3},a_{4})$, where
$0\le j_k\le m$ and $1\le a_k<j_k/2+C$.
Write
\begin{align*}
 n(\vec j,\vec a)=\#\cQ_{0}^{a,b}(\vec j,\vec a),
 \quad H(\vec j)=\tfrac12\sum_{k=1}^{4} j_k,
\end{align*}
and read all subscripts cyclically. Lemma~\ref{lemma:Medium-line} gives
\begin{equation}
\label{equ:app-eight}
 n(\vec j,\vec a)\lesssim 2^{E_{k,l}},
 \quad
 E_{k,l}=2j_k-2a_k+a_{k+1}+a_{k+1+l},
 \quad 1\le k\le4,\quad l=1,2.
\end{equation}
Consequently $n\lesssim2^{\sum_{k,l}c_{k,l}E_{k,l}}$ for every
nonnegative array $c$ of total sum one. Taking respectively
\begin{align*}
 c=\frac1{24}
 \begin{pmatrix}2&3\\3&4\\0&6\\3&3\end{pmatrix},
 \quad
 c=\frac1{12}
 \begin{pmatrix}1&2\\1&2\\3&0\\1&2\end{pmatrix}
\end{align*}
yields
\begin{align*}
 n\lesssim2^{H-(j_{1}-j_{2})/12},
 \quad n\lesssim2^{H+(a_{1}-a_{2})/6}.
\end{align*}
Rotations and reflections of the vertex labels give the same bounds
for either sign of each adjacent difference. Taking the geometric mean
of the four bounds chosen to give decay in the $j$-differences and
the four chosen to give decay in the $a$-differences proves
\begin{equation}
\label{equ:app-decay}
 n(\vec j,\vec a)\lesssim2^{H-\delta D(\vec j,\vec a)},
 \quad
 D(\vec j,\vec a)=\sum_{k=1}^{4}
 (|j_k-j_{k+1}|+|a_k-a_{k+1}|),
\end{equation}
for a fixed sufficiently small $\delta>0$.
Multiplying by $2^{-H}\prod_k\lambda_{j_k}$ and summing gives
\begin{align*}
 \sum_{Q\in\cQ_{0,2,2}^{a,b}}f(Q)
 \lesssim
 \sum_{\vec j,\vec a}2^{-\delta D(\vec j,\vec a)}
 \prod_{k=1}^{4}\lambda_{j_k}
 \lesssim m\n\lambda_{\ell^{2}}^{4}.
\end{align*}
Indeed, with $a_{4}$ fixed, the sum over $a_{1},a_{2},a_{3}$ is uniformly
bounded by three geometric sums. There are $O(m)$ possible values of
$a_{4}$. In the $j$-sum, retaining just the differences
$|j_{1}-j_{2}|$ and $|j_{3}-j_{4}|$ permits two applications of
\eqref{equ:app-young}. This proves \eqref{equ:type22}.

To keep track of the divisor weight, set
\begin{align*}
 W(\vec j,\vec a)=
 \sum_{Q\in\cQ_{0}^{a,b}(\vec j,\vec a)}
 \frac1{d_{a,b}(\xi_{1}-\xi_{4})}.
\end{align*}
Since $d_{a,b}(\xi_{1}-\xi_{4})\ge1$, all eight bounds
$W\lesssim2^{E_{k,l}}$ hold. For $(k,l)=(1,2)$,
Lemma~\ref{lemma:divisor side} improves this to
\begin{align*}
 W\lesssim2^{E_{1,2}-a_{4}/2}.
\end{align*}
Taking the geometric mean of these eight bounds, using this improvement
for the $(1,2)$ term and observing that
$\frac18\sum_{k,l}E_{k,l}=H$, gives
\begin{align*}
 W\lesssim2^{H-a_{4}/16}.
\end{align*}
Combining this with $W\le n$ and \eqref{equ:app-decay} yields
\begin{equation}
\label{equ:app-weighted-decay}
 W(\vec j,\vec a)
 \lesssim 2^{H-\frac\delta2D(\vec j,\vec a)-a_{4}/32}.
\end{equation}
After multiplication by $2^{-H}\prod_k\lambda_{j_k}$, the same
summation as above now has the convergent last sum
$\sum_{a_{4}\ge1}2^{-a_{4}/32}$. Thus
\begin{align*}
 \sum_{Q\in\cQ_{0,2,2}^{a,b}}
 \frac{f(Q)}{d_{a,b}(\xi_{1}-\xi_{4})}
 \lesssim\n\lambda_{\ell^{2}}^{4},
\end{align*}
which proves \eqref{equ:type22-div}.

\medskip\noindent
\emph{Case 3: $(\alpha,\beta)=(3,3)$.}
Let $q_{\vec j}$ count these rectangles with $\xi_k\in S_{j_k}$.
All four richness indices are zero. Lemma~\ref{lemma:corner}, with
cyclic relabeling, gives
\begin{align*}
 q_{\vec j}\lesssim 2^{\min_k(j_k+j_{k+1})}
 =2^{H-\frac12(|j_{1}-j_{3}|+|j_{2}-j_{4}|)}.
\end{align*}
The equality follows by writing the four adjacent sums minus $H$
as the four values $\frac12(\pm(j_{1}-j_{3})\pm(j_{2}-j_{4}))$.
Multiplication by $2^{-H}\prod_k\lambda_{j_k}$, followed by
\eqref{equ:app-young} for the pairs $(j_{1},j_{3})$ and $(j_{2},j_{4})$,
proves \eqref{equ:type,off} in this case.

\medskip\noindent
\emph{Case 4: $(\alpha,\beta)=(2,3)$ or $(3,2)$.}
By relabeling it suffices to consider $(2,3)$. Write $q_{\vec j}$
for the count at fixed $\vec j$. The one-corner bound at the poor
vertices gives
\begin{equation}
\label{equ:app-mixed-corner}
 q_{\vec j}\lesssim2^{\min\{j_{1}+j_{4},j_{2}+j_{3},j_{1}+j_{2}\}}.
\end{equation}
At fixed $\vec a=(a_{1},a_{2},0,0)$, Lemmas
\ref{lemma:corner} and \ref{lemma:Medium-line}, allowing interchange of the
two neighbors of a vertex, give
\begin{align*}
 n(\vec j,\vec a)\lesssim2^{j_{3}+j_{4}+(a_{1}+a_{2})/2}, \quad
 n(\vec j,\vec a)\lesssim2^{j_{1}+j_{2}-(a_{1}+a_{2})}.
\end{align*}
Taking the first estimate to power $3/5$ and the second to power
$2/5$, and summing the resulting geometric series in $a_{1},a_{2}$,
yields
\begin{equation}\label{equ:app-mixed-sum}
 q_{\vec j}\lesssim2^{\frac25(j_{1}+j_{2})+\frac35(j_{3}+j_{4})}.
\end{equation}
For clarity, the decay obtained by combining these bounds can be seen
without further interpolation bookkeeping. Put
$x=j_{1}-j_{3}$, $y=j_{2}-j_{4}$. The four exponents in
\eqref{equ:app-mixed-corner}--\eqref{equ:app-mixed-sum}, minus $H$,
are
\begin{align*}
 \tfrac12(x-y),\quad-\tfrac12(x-y),\quad
 \tfrac12(x+y),\quad-\tfrac1{10}(x+y).
\end{align*}
Their minimum is at most
$-\frac1{10}\max\{|x-y|,|x+y|\}
 =-\frac1{10}(|x|+|y|)$.
Hence
\begin{align*}
 q_{\vec j}\lesssim
 2^{H-\frac1{10}(|j_{1}-j_{3}|+|j_{2}-j_{4}|)}.
\end{align*}
Two applications of \eqref{equ:app-young}, as in Case 3, prove
\eqref{equ:type,off}. This completes all four cases.
\end{proof}

\end{document}